\documentclass[smallcondensed]{amsart}

\usepackage{amssymb}
\usepackage{mathtext}
\usepackage{amsmath}
\usepackage{mathtools}
\usepackage[all]{xy}
\usepackage{tikz}
\usetikzlibrary{calc}
\usepgflibrary{shapes.geometric}
\usepgflibrary{shapes.misc}
\usetikzlibrary{positioning}
\usetikzlibrary{decorations}
\usetikzlibrary{decorations.pathreplacing}
\usetikzlibrary{shapes,snakes}
\usepackage{url}
\usepackage[T1]{fontenc}
\usepackage[english]{babel}
\usepackage{verbatim}
\usepackage{mathrsfs}
\usepackage{leftidx}
\usepackage{pgfplots}
\usepackage[eulergreek]{sansmath}
\usepackage{fix-cm}
\usepackage{lmodern}

\newcommand{\Hom}{\operatorname{Hom}\nolimits}

\newcommand{\Ext}{\operatorname{Ext}\nolimits}
\newcommand{\AR}{\operatorname{AR}\nolimits}
\newcommand{\Mod}{\operatorname{mod}\nolimits}

\newcommand{\nc}{\operatorname{nc}\nolimits}

\newcommand{\id}{\operatorname{id}\nolimits}

\newcommand{\ZZ}{\mathbb{Z}}

\newcommand{\CC}{\mathcal{C}}
\newcommand{\X}{\mathscr{X}}

\newcommand{\A}{\mathscr{A}}
\newcommand{\p}{\mathcal{P}}
\newcommand{\D}{\mathscr{D}}

\theoremstyle{plain}
\newtheorem{lemma}{Lemma}
\theoremstyle{plain}
\theoremstyle{definition}
\newtheorem{lm}{Lemma}[section]
\newtheorem{corollary}[lm]{Corollary}
\newtheorem{theorem}[lm]{Theorem}
\newtheorem{definition}[lm]{Definition}

\newtheorem{example}[lm]{Example}

\theoremstyle{remark}

\theoremstyle{remark}
\newtheorem{remark}[lm]{Remark}

\begin{document}

\title{Mutation of torsion pairs in cluster categories of Dynkin type $D$
}

\thanks{This work has been carried out in the framework of the research priority programme SPP 1388 Darstellungstheorie of the Deutsche Forschungsgemeinschaft (DFG). The author gratefully acknowledges financial support through the grant HO 1880/5-1.}

\author{Sira Gratz         
}



\maketitle

\begin{abstract}
In cluster categories, mutation of torsion pairs provides a generalisation for the mutation of cluster tilting subcategories, which models the combinatorial structure of cluster algebras. In this paper we present a geometric model for mutation of torsion pairs in the cluster category $\CC_{D_n}$ of Dynkin type $D_n$. Using a combinatorial model introduced by Fomin and Zelevinsky in \cite{FZ03}, subcategories in $\CC_{D_n}$ correspond to rotationally invariant collections of arcs in a regular $2n$-gon, which we call diagrams of Dynkin type $D_n$. Torsion pairs in $\CC_{D_n}$ have been classified by Holm, J\o rgensen and Rubey in \cite{HJR}. They correspond to so-called Ptolemy diagrams of Dynkin type $D_n$, which are diagrams of Dynkin type $D_n$ satisfying a certain  combinatorial condition. We define mutation of a diagram $\X$ of Dynkin type $D_n$ with respect to a compatible diagram $\D$ of Dynkin type $D_n$ consisting of pairwise non-crossing arcs. Such a diagram $\D$ partitions the regular $2n$-gon into cells and mutation of $\X$ with respect to $\D$ can be thought of as a rotation of each of these cells. We show that mutation of Ptolemy diagrams of Dynkin type $D_n$ corresponds to mutation of the corresponding torsion pairs in the cluster category of Dynkin type $D_n$.
\end{abstract}

\section{Introduction}
\label{intro}
An important motivation for mutation stems from cluster theory. Cluster algebras were introduced by Fomin and Zelevinsky in \cite{FZ} to study dual canonical bases and total positivity in semisimple groups from an algebraic and combinatorial viewpoint. In \cite{BMRRT}, Buan, Marsh, Reineke, Reiten and Todorov categorified cluster algebras through cluster categories. By Keller \cite{Keller}, cluster categories are triangulated categories. They mimic the combinatorial idea of cluster algebras: the rigid indecomposable objects correspond to the cluster variables and so-called cluster tilting subcategories take on the role of the clusters. The mutation of clusters  is modelled by the exchange of indecomposable objects in cluster tilting subcategories as introduced by Buan, Iyama, Reiten and Scott in \cite{BIRS}: one replaces one indecomposable object in the cluster tilting subcategory by a unique other indecomposable object such that we again get a cluster tilting subcategory. 

In \cite{IY}, Iyama and Yoshino introduced a more general concept of mutation in triangulated categories. Every subcategory $X$ of $T$ which is closed under summands, finite direct sums and isomorphisms can be mutated in two directions with respect to a rigid subcategory $D \subset T$, i.e.\;a subcategory that has no self-extension; $\Ext^1(D, D) = 0$ -- yielding two subcategories $\mu_D(X)$ and $\mu^-_D(X)$. The exchange of indecomposable objects in cluster tilting subcategories is a special case of this notion of mutation in a triangulated category.

Not all triangulated categories have cluster tilting subcategories, see\;for example Burban, Iyama, Keller and Reiten's paper \cite{BIKR}. However, they always contain torsion pairs, which were introduced for triangulated categories by Iyama and Yoshino in \cite{IY}. A torsion pair in a triangulated category $T$ is a pair of subcategories $(X,Y)$, such that there are no morphisms from $X$ to $Y$ and every object in $T$ can be written as an extension of an object in $Y$ by an object in $X$. Any torsion pair $(X,Y)$ is defined uniquely by the subcategory $X$, which is called the torsion part or equivalently by the subcategory $Y$, which is called the torsion-free part. Cluster tilting subcategories can be viewed as a specialisation of torsion pairs, as every cluster tilting subcategory in $T$ is the torsion part of a torsion pair in $T$.

A natural question to ask is how to define mutation of torsion pairs to provide a generalisation of mutation of cluster tilting subcategories. It was shown by Zhou and Zhu in \cite{ZZ} that if $T$ has Auslander-Reiten triangles, then mutation of a torsion pair of $T$ in the sense of Iyama and Yoshino \cite{IY} with respect to a suitably nice rigid subcategory $D$ in both directions yields another torsion pair of $T$. Furthermore they used the classification of torsion pairs in the cluster category of type $A_{\infty}$ due to Ng \cite{Ng} and of Dynkin type $A_n$ due to Holm, J\o rgensen and Rubey \cite{HJRAn} via Ptolemy diagrams to give a combinatorial description of their mutation in these cases.

In \cite{HJR}, Holm, J\o rgensen and Rubey classified torsion pairs in the cluster category of Dynkin type $D_n$ using Ptolemy diagrams of Dynkin type $D_n$. They used a combinatorial model for Dynkin type $D_n$ which was first introduced by Fomin and Zelevinsky in \cite{FZ03} and which is closely related to the model for the cluster category of Dynkin type $D_n$ by Schiffler \cite{Schiffler}. Consider the regular $2n$-gon $\p_{2n}$. An arc of $\p_{2n}$ is a pair of vertices of $\p_{2n}$ and an arc that is invariant under rotation by $\pi$ is called a diameter. Indecomposable objects in $\CC_{D_n}$ are identified with rotationally symmetric pairs of arcs and diameters of the colours red and green in the regular $2n$-gon. Subcategories which are closed under summands, finite direct sums and isomorphisms correspond to $\pi$-rotation invariant collections of arcs, which we call diagrams of Dynkin type $D_n$. It was shown by Holm, J\o rgensen and Rubey in \cite{HJR} that torsion parts of torsion pairs in $\CC_{D_n}$ correspond to diagrams of Dynkin type $D_n$ with a distinctive combinatorial property, called Ptolemy diagrams of Dynkin type $D_n$.

In this paper we provide a combinatorial description of mutation of torsion pairs in the cluster category $\CC_{D_n}$ of Dynkin type $D_n$.  The situation is more complicated than in types $A_n$ and $A_{\infty}$, because we have to deal with the indecomposable objects in the cluster category which arise from the exceptional vertices of the Dynkin diagram $D_n$. Combinatorially, these correspond to the red and green diameters in the $2n$-gon. We define the mutation of Ptolemy diagrams of Dynkin type $D_n$. Just like mutation of a torsion pair of $\CC_{D_n}$ is with respect to a nice subcategory $D$ of $\CC_{D_n}$, the mutation of the corresponding Ptolemy diagram of Dynkin type $D_n$ is with respect to the subdiagram $\D$ corresponding to the subcategory $D$. Any such subdiagram $\D$ consists of pairwise non-crossing arcs and divides the $2n$-gon into convex polygons which we call $\D$-cells of Dynkin type $D_n$. We define the mutations $\mu_{\D}$ and $\mu^-_{\D}$ as mutually inverse bijections on arcs that do not cross any arcs of $\D$. Thus we can mutate any diagram $\X$ of Dynkin type $D_n$ with respect to any non-crossing diagram $\D$ which is compatible in the sense that no arc in $\X$ crosses any arc in $\D$. Essentially, the mutations $\mu_{\D}$ and $\mu^-_{\D}$ can be thought of as rotating the arcs within each of the $\D$-cells of Dynkin type $D_n$ in a clockwise respectively anti clockwise direction. We show that  mutation of Ptolemy diagrams of Dynkin type $D_n$ provides a combinatorial model for mutation of torsion pairs in the cluster category $\CC_{D_n}$.

The paper is organized as follows. In Section \ref{Combinatorial model} we give an overview over the combinatorial model for the cluster category of Dynkin type $D_n$ as introduced by Fomin and Zelevinsky in \cite{FZ03}. In Section \ref{Torsion pairs} we review the definition of a torsion pair from Iyama and Yoshino's paper \cite{IY}. Furthermore, we recall the classification of torsion pairs in the cluster category $\CC_{D_n}$ of Dynkin type $D_n$ due to Holm, J\o rgensen and Rubey \cite{HJR}. Section \ref{Mutation of torsion pairs} provides an overview over mutation in triangulated categories as introduced by Iyama and Yoshino in \cite{IY} and we recall the main result about mutation of torsion pairs from Zhou and Zhu's paper \cite{ZZ}. In Section \ref{Non-crossing diagrams} we discuss non-crossing diagrams of Dynkin type $D_n$ and provide the combinatorics needed for the definition of mutation of Ptolemy diagrams of Dynkin type $D_n$ in Section \ref{A combinatorial model for mutation}. Moreover, Section \ref{A combinatorial model for mutation} contains the proof of our main result which states that mutation of a torsion pair $(X,Y)$ in the cluster category $\CC_{D_n}$ with respect to a suitable subcategory $D$ of $\CC_{D_n}$ corresponds to mutation of the Ptolemy diagram $\X$ of Dynkin type $D_n$ associated to $X$ with respect to the diagram $\D$ associated to $D$.

\paragraph{Acknowledgements}

The author thanks her supervisor, Thorsten Holm, as well as David Pauksztello for helpful discussions and the advice they have given during the preparation of this paper.

\paragraph{Conventions}

Throughout this paper we work over an algebraically closed field $k$. All triangulated categories are assumed to be $k$-linear, Hom-finite and Krull-Schmidt. All subcategories are assumed to be full and closed under isomorphisms, direct summands and finite direct sums. By abuse of notation we might identify a subcategory $X$ with the collection of its (indecomposable) objects and write $x \in X$ for an object $x$ in $X$. For a subcategory $X$ of a triangulated category $T$, we denote by $X^\perp$ the right Hom-perp of $X$, i.e.\;the subcategory $X^\perp := \{t \in  T| \Hom(x,t) = 0\; \forall x \in X\}$ and dually by $^\perp X$ the subcategory $^\perp X:= \{t \in  T| \Hom(t,x) = 0\; \forall x \in X\}$. When we refer to the Dynkin diagram $D_n$ or related combinatorial concepts, we will always assume $n \geq 4$.

\section{A combinatorial model for the cluster category of Dynkin type $D_n$}
\label{Combinatorial model}

The combinatorial model for Dynkin type $D_n$ introduced by Fomin and Zelevinsky in \cite{FZ03} offers a geometric interpretation of the cluster category of Dynkin type $D_n$. Isomorphism classes of indecomposable objects are represented by rotation-invariant pairs of arcs and diameters in a regular $2n$-gon. As a useful property, this combinatorial model allows for an easy way to determine the dimension of the extension space between two indecomposable objects by counting the number of times the corresponding pairs of arcs or diameters in the regular $2n$-gon cross. First we recall what the Auslander-Reiten quiver of the cluster category of Dynkin type $D_n$ looks like, such as to describe the one-to-one correspondence between isomorphism classes of indecomposable objects, represented by vertices in the Auslander-Reiten quiver, and pairs of arcs and diameters in the regular $2n$-gon.

We denote by $\CC_{D_n}$ the cluster category of Dynkin type $D_n$. It is defined as the orbit category \[D^b(kD_n)/\tau^{-1} \Sigma,\] where $D^b(kD_n)$ is the bounded derived category of finitely generated $kD_n$-modules with Auslander-Reiten translation $\tau$ and shift functor $\Sigma$. By Buan, Marsh, Reineke, Reiten and Todorov \cite{BMRRT} the cluster category $\CC_{D_n}$ is a $k$-linear Hom-finite 2-Calabi-Yau Krull-Schmidt category and by Keller \cite{Keller} it is triangulated with shift functor $\Sigma$. By Happel (\cite{Happel}, Corollary 4.5(i)) the Auslander-Reiten quiver $\AR(D^b(kD_n))$ of the bounded derived category $D^b(kD_n)$ is the repetition quiver $\ZZ D_n$. We label the vertices with the coordinate system first introduced by Iyama in \cite{Iyama}, Definition 4.2 (cf.\;Figure \ref{fig:c-system}). 

The functor $\tau^{-1}\Sigma$ is an auto-equivalence of $D^b(kD_n)$ and thus acts on its set of isomorphism classes of indecomposable objects, i.e.\;on the vertices of the Auslander-Reiten quiver $\AR(D^b(kD_n))$. By Table 1 in Miyachi and Yekutieli's paper \cite{MY},  it is defined on the vertices of $\AR(D^b(kD_n))$ by
$$
\tau^{-1}\Sigma: \begin{cases}
[i,j] \mapsto [i+n,j+n] \text{ for i<j<i+n,}\\
[i,i+n]_{\pm} \mapsto \begin{cases}
[i+n,i+2n]_\pm \text{ if $n$ is even,}\\
[i+n,i+2n]_\mp \text{ if $n$ is odd.}
\end{cases} 
\end{cases}
$$
Identifying the vertices of $\AR(kD_n)$ in the orbits of $\tau^{-1}\Sigma$ gives rise to the Auslander Reiten quiver $\AR(\CC_{D_n})$ of the cluster category of Dynkin type $D_n$. We use the coordinate system induced from the one on $\AR(D^b(kD_n))$ to label the vertices of $\AR(\CC_{D_n})$, where the sections have first coordinates $0$ up to $n-1$, cf.\;Figure \ref{fig:AR(C)}. By abuse of notation we will sometimes omit the index and just write $[i,i+n]$ for either of the vertices $[i,i+n]_+$ and $[i,i+n]_-$, if the sign is not important in the context.

\begin{centering}
\begin{figure}
\begin{tikzpicture} [scale=1.4, font = \footnotesize, font = \sffamily, font=\sansmath\sffamily]

\node (a1) at (0,0) {$[0,2]$};
\node (a2) at (1,1) {$[0,3]$};
\node (af) at (3,3) {$\ldots$};
\node (a3) at (2,2) {$[0,4]$};
\node (a+) at (4,4) {$[0,n]_+$};
\node (a-) at (4,3) {$[0,n]_-$};

\draw[->] (a1) -- (a2);
\draw[->] (a2) -- (a3);
\draw[dotted] (a3) -- (af);
\draw[->] (af) -- (a+);
\draw[->] (af) -- (a-);

\node (b1) at (2,0) {$[1,3]$};
\node (b2) at (3,1) {$[1,4]$};
\node (bf) at (5,3) {$\ldots$};
\node (b3) at (4,2) {$[1,5]$};
\node (b+) at (6,4) {$[1,n+1]_+$};
\node (b-) at (6,3) {$[1,n+1]_-$};

\draw[->] (b1) -- (b2);
\draw[->] (b2) -- (b3);
\draw[dotted] (b3) -- (bf);
\draw[->] (bf) -- (b+);
\draw[->] (bf) -- (b-);

\node (c1) at (4,0) {$[2,4]$};
\node (c2) at (5,1) {$[2,5]$};
\node (cf) at (7,3) {$\ldots$};
\node (c3) at (6,2) {$[2,6]$};
\node (c+) at (8,4) {$[2,n+2]_+$};
\node (c-) at (8,3) {$[2,n+2]_-$};

\draw[->] (c1) -- (c2);
\draw[->] (c2) -- (c3);
\draw[dotted] (c3) -- (cf);
\draw[->] (cf) -- (c+);
\draw[->] (cf) -- (c-);

\draw[->] (a2) -- (b1);
\draw[->] (b2) -- (c1);
\draw[->] (a3) -- (b2);
\draw[->] (b3) -- (c2);
\draw[->] (a+) -- (bf);
\draw[->] (a-) -- (bf);
\draw[->] (b+) -- (cf);
\draw[->] (b-) -- (cf);

\draw[dotted] (af) -- (b3);
\draw[dotted] (bf) -- (c3);

\draw[dotted] (0,2) -- (a2);
\draw[dotted] (1,3) -- (a3);
\draw[dotted] (2,4) -- (af);

\draw[dotted] (c2) -- (6,0);
\draw[dotted] (c3) -- (7,1);
\draw[dotted] (cf) -- (8,2);

\end{tikzpicture}
\caption{The coordinate system on $\AR(D^b(kD_n))$}
\label{fig:c-system}
\end{figure}
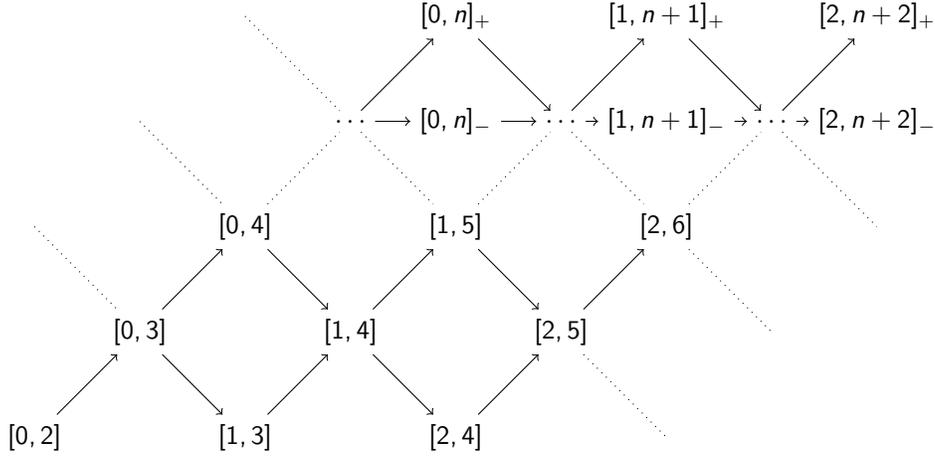
\end{centering}

\begin{centering}
\begin{figure}
\begin{tikzpicture} [scale=1.4, font = \footnotesize, font=\sansmath\sffamily
, font=\sansmath\sffamily
]

\node (b1) at (2,0) {$[0,2]$};
\node (b2) at (3,1) {$[0,3]$};
\node (bf) at (5,3) {$\ldots$};
\node (b3) at (4,2) {$[0,3]$};
\node (b+) at (6,4) {$[0,n]_+$};
\node (b-) at (6,3)  {$[0,n]_-$};

\draw[->] (b1) -- (b2);
\draw[->] (b2) -- (b3);
\draw[dotted] (b3) -- (bf);
\draw[->] (bf) -- (b+);
\draw[->] (bf) -- (b-);

\node (c1) at (4,0) {$\bullet$};
\node (c2) at (5,1) {$\bullet$};
\node (cf) at (7,3) {$\bullet$};
\node (c3) at (6,2) {$\bullet$};
\node (c+) at (8,4) {$\bullet$};
\node (c-) at (8,3) {$\bullet$};

\node (d1) at (8,0) {$[0,2]$};
\node (d2) at (9,1) {$[0,3]$};
\node (df) at (11,3) {$\ldots$};
\node (d3) at (10,2) {$[0,4]$};
\node (d+) at (12,4) {$[0,n]_+$};
\node (d-) at (12,3) {$[0,n]_-$};

\draw[->] (d1) -- (d2);
\draw[->] (d2) -- (d3);
\draw[dotted] (d3) -- (df);
\draw[->] (df) -- (d+);
\draw[->] (df) -- (d-);

\node (e1) at (6,0) {$[n-1,n+1]$};
\node (e2) at (7,1) {$[n-1,n+2]$};
\node (ef) at (9,3) {$\ldots$};
\node (e3) at (8,2) {$[n-1,n+3]$};
\node (e+) at (10,4) {$[n-1,2n-1]_+$};
\node (e-) at (10,3) {$[n-1,2n-1]_-$};

\draw[->] (e1) -- (e2);
\draw[->] (e2) -- (e3);
\draw[dotted] (e3) -- (ef);
\draw[->] (ef) -- (e+);
\draw[->] (ef) -- (e-);

\draw[->] (c1) -- (c2);
\draw[->] (c2) -- (c3);
\draw[dotted] (c3) -- (cf);
\draw[->] (cf) -- (c+);
\draw[->] (cf) -- (c-);

\draw[->] (b2) -- (c1);
\draw[->] (b3) -- (c2);
\draw[->] (b+) -- (cf);
\draw[->] (b-) -- (cf);

\draw[dotted] (bf) -- (c3);

\draw[dotted] (c2) -- (6,0);
\draw[dotted] (c3) -- (7,1);
\draw[dotted] (cf) -- (8,2);

\draw[dotted] (c2) -- (6,0);
\draw[dotted] (c3) -- (7,1);
\draw[dotted] (cf) -- (8,2);

\draw[dotted] (c-) -- (ef);
\draw[dotted] (c+) -- (ef);

\draw[->] (e2) -- (d1);
\draw[->] (e3) -- (d2);
\draw[dotted] (ef) -- (d3);
\draw[->] (e-) -- (df);
\draw[->] (e+) -- (df);

\end{tikzpicture}
\caption{The coordinate system on $\AR(\CC_{D_n})$}
\label{fig:AR(C)}
\end{figure}
\end{centering}

Consider now the regular $2n$-gon $\p_{2n}$ with vertices labelled consecutively in an anti clockwise direction by $0, 1$, $\ldots$, $2n-1$. Throughout we will calculate modulo $2n$. The isomorphism classes of indecomposable objects in $\CC_{D_n}$ correspond to so-called pairs of arcs and coloured diameters in $\p_{2n}$.
\begin{definition}
An {\em arc} of $\p_{2n}$ is a pair of vertices $(i,j)$ of $\p_{2n}$ with $i \neq j$. An arc of the form $(i,i+1)$, for $i = 0, \ldots, 2n-1$, is called an {\em edge} of $\p_{2n}$. An {\em internal arc} of $\p_{2n}$ is an arc of $\p_{2n}$ that is not an edge. Each arc $(i,j)$ of $\p_{2n}$ has a partner $(i+n,j+n)$ which is obtained from $(i,j)$ by rotation by $\pi$. An arc of $\p_{2n}$ is called {\em diameter} if it is $\pi$-rotation invariant, i.e.\;if it is of the form $(i,i+n)$. A non-diameter arc $(i,j)$ together with its partner $(i+n,j+n)$ is called a {\em pair of arcs} and denoted by $\overline{(i,j)}$. For each diameter $(i,i+n)$ we introduce two coloured diameters: a red one $\overline{(i,i+n)}_r$ and a green one $\overline{(i,i+n)}_g$. By abuse of notation we sometimes omit the index and just write $\overline{(i,i+n)}$ for a coloured diameter, which could be either red or green. If we omit the overline and simply write $(i,i+n)$, we refer to the diameter $(i,i+n)$ without a colour. We set
$$
\mathcal{E}(\p_{2n}) := \{\text{pairs of edges of }\p_{2n}\}
$$
and
$$
\A(\p_{2n}) := \{\text{ pairs of arcs and coloured diameters of }\p_{2n}\} \setminus \mathcal{E}(\p_{2n}).
$$
\end{definition}

\begin{definition}\label{D:bijection}
Denote by $\AR(\CC_{D_n})_0$ the set of vertices of the Auslander Reiten quiver $\AR(\CC_{D_n})$.
Let $b: \AR(\CC_{D_n})_0 \to \A(\p_{2n})$ be the bijection defined by
$$
b: \begin{cases}
[i,j] \mapsto \overline{(i,j)} \text{ for $0 \leq i<j<i+n$,}\\
[i,i+n]_+ \mapsto \begin{cases}
\overline{(i,i+n)}_g \text{ if $i = 0 \mod 2$,}\\
\overline{(i,i+n)}_r \text{ if $i = 1 \mod 2$,}
\end{cases} \text{ for } 0 \leq i \leq n-1 \\
[i,i+n]_- \mapsto \begin{cases}
\overline{(i,i+n)}_r \text{ if $i = 0 \mod 2$,}\\
\overline{(i,i+n)}_g \text{ if $i = 1 \mod 2$}
\end{cases} 
\text{ for } 0 \leq i \leq n-1.
\end{cases}
$$
Via this bijection, we can associate to each pair of arcs or coloured diameter $\overline{(i,j)} \in \A(\p_{2n})$ an indecomposable object $m_{\overline{(i,j)}}$ in $\CC_{D_n}$, which is unique up to isomorphism. The vertex $[i,j]=b(\overline{(i,j)})$ in the Auslander-Reiten quiver $\AR(\CC_{D_n})$ represents the isomorphism class of the module $m_{\overline{(i,j)}}$. 
\end{definition}

The map $b$ defined in Definition \ref{D:bijection} provides the connection to the combinatorial model for Dynkin type $D_n$. It sends vertices without a sign in $\AR(\CC_{D_n})$ to pairs of internal arcs of $\p_{2n}$ and alternatingly matches the vertices $[i,i+n]_{\pm}$ to red and green diameters. Figure \ref{fig:matching} illustrates the matching of vertices with a sign to coloured diameters.

\begin{centering}
\begin{figure}
\begin{tikzpicture} [scale=1, font = \footnotesize, font=\sansmath\sffamily
]
\tikzstyle{ann} = [draw=none,fill=none,right]
\node (a1) at (0,0) {$\bullet$};
\node (a2) at (1,1) {$\bullet$};
\node (af) at (3,3) {$\bullet$};
\node (a3) at (2,2) {$\bullet$};
\node[draw, circle, dashed] (a+) at (4,4) {$\bullet$};
\node[draw, circle] (a-) at (4,3) {$\bullet$};

\draw[->] (a1) -- (a2);
\draw[->] (a2) -- (a3);
\draw[dotted] (a3) -- (af);
\draw[->] (af) -- (a+);
\draw[->] (af) -- (a-);

\node (b1) at (2,0) {$\bullet$};
\node (b2) at (3,1) {$\bullet$};
\node (bf) at (5,3) {$\bullet$};
\node (b3) at (4,2) {$\bullet$};
\node[draw, circle] (b+) at (6,4) {$\bullet$};
\node[draw, circle, dashed] (b-) at (6,3) {$\bullet$};

\draw[->] (b1) -- (b2);
\draw[->] (b2) -- (b3);
\draw[dotted] (b3) -- (bf);
\draw[->] (bf) -- (b+);
\draw[->] (bf) -- (b-);

\node (c1) at (4,0) {$\bullet$};
\node (c2) at (5,1) {$\bullet$};
\node (cf) at (7,3) {$\bullet$};
\node (c3) at (6,2) {$\bullet$};
\node[draw, circle, dashed] (c+) at (8,4) {$\bullet$};
\node[draw, circle] (c-) at (8,3) {$\bullet$};

\node (d1) at (8,0) {$\bullet$};
\node (d2) at (9,1) {$\bullet$};
\node (df) at (11,3) {$\bullet$};
\node (d3) at (10,2) {$\bullet$};
\node[draw, circle, dashed] (d+) at (12,4) {$\bullet$};
\node[draw, circle] (d-) at (12,3) {$\bullet$};

\draw[->] (d1) -- (d2);
\draw[->] (d2) -- (d3);
\draw[dotted] (d3) -- (df);
\draw[->] (df) -- (d+);
\draw[->] (df) -- (d-);

\node (e1) at (6,0) {$\bullet$};
\node (e2) at (7,1) {$\bullet$};
\node (ef) at (9,3) {$\bullet$};
\node (e3) at (8,2) {$\bullet$};
\node[draw, circle] (e+) at (10,4) {$\bullet$};
\node[draw, circle, dashed] (e-) at (10,3) {$\bullet$};

\draw[->] (e1) -- (e2);
\draw[->] (e2) -- (e3);
\draw[dotted] (e3) -- (ef);
\draw[->] (ef) -- (e+);
\draw[->] (ef) -- (e-);

\draw[->] (c1) -- (c2);
\draw[->] (c2) -- (c3);
\draw[dotted] (c3) -- (cf);
\draw[->] (cf) -- (c+);
\draw[->] (cf) -- (c-);

\draw[->] (a2) -- (b1);
\draw[->] (b2) -- (c1);
\draw[->] (a3) -- (b2);
\draw[->] (b3) -- (c2);
\draw[->] (a+) -- (bf);
\draw[->] (a-) -- (bf);
\draw[->] (b+) -- (cf);
\draw[->] (b-) -- (cf);

\draw[dotted] (af) -- (b3);
\draw[dotted] (bf) -- (c3);

\draw[dotted] (c2) -- (6,0);
\draw[dotted] (c3) -- (7,1);
\draw[dotted] (cf) -- (8,2);

\draw[dotted] (c2) -- (6,0);
\draw[dotted] (c3) -- (7,1);
\draw[dotted] (cf) -- (8,2);

\draw[dotted] (c-) -- (ef);
\draw[dotted] (c+) -- (ef);

\draw[->] (e2) -- (d1);
\draw[->] (e3) -- (d2);
\draw[dotted] (ef) -- (d3);
\draw[->] (e-) -- (df);
\draw[->] (e+) -- (df);

\end{tikzpicture}
\caption{The vertices in $\AR(\CC_{D_n})$ marked with a dashed circle get matched to green diameters, those marked with a continuous circle get matched to red diameters}
\label{fig:matching}
\end{figure}
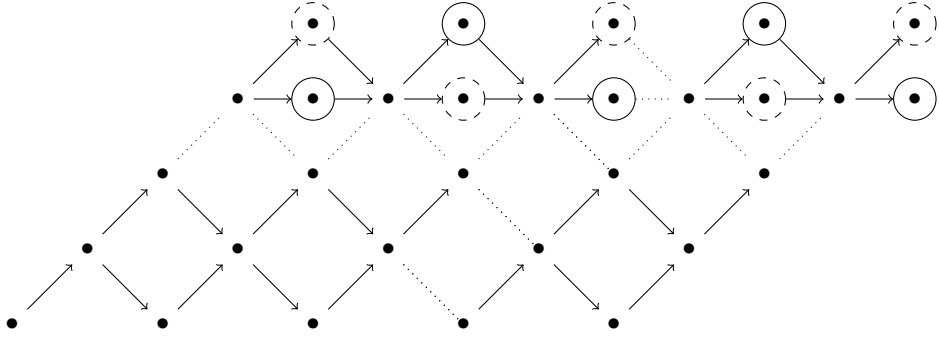
\end{centering}

\begin{definition}[diagram of Dynkin type $D_n$]
A subset of $\A(\p_{2n})$ is called a {\em diagram of Dynkin type $D_n$}.
\end{definition}

\begin{remark}\label{rotation-invariance}
\begin{itemize}
\item
{Any diagram of Dynkin type $D_n$ is invariant under rotation by $\pi$. }

\item
{The bijection $b: \AR(\CC_{D_n})_0 \to \A(\p_{2n})$ induces a one-to-one correspondence between subcategories of $\CC_{D_n}$ and diagrams of Dynkin type $D_n$, where the subcategory consisting of the zero object is assumed to correspond to the empty diagram.}
\end{itemize}
\end{remark}

\begin{definition}\label{crossing arcs}
\begin{itemize}
\item{Two arcs $(i,j)$ and $(k,l)$ are said to {\em cross}, if $i,j,k$ and $l$ are pairwise distinct and they lie on the boundary of $\p_{2n}$ in the order $i,k,j,l$ or $k,i,l,j$ when moving in an anti clockwise direction.}
\item{The pairs of arcs $\overline{(i,j)}$ and $\overline{(k,l)}$ are said to {\em cross once}, if the arc $(i,j)$ crosses either $(k,l)$ or $(k+n,l+n)$. They are said to {\em cross twice}, if the arc $(i,j)$ crosses both $(k,l)$ and $(k+n,l+n)$.}
\item{A coloured diameter $\overline{(i,i+n)}$ and a pair of arcs $\overline{(k,l)}$ are said to {\em cross (once)}, if the arc $(i,i+n)$ crosses $(k,l)$.}
\item{Two coloured diameters $\overline{(i,i+n)}_r$ and $\overline{(j,j+n)}_g$ of different colours are said to {\em cross (once)} if the arcs $(i,i+n)$ and $(j,j+n)$ cross. Two diameters of the same colour do not cross.}
\end{itemize}
Figure \ref{fig:crossing} illustrates the crossing of arcs. Throughout this paper, when drawing diagrams we will draw green diameters as wriggly lines and red diameters as straight lines.
\end{definition}

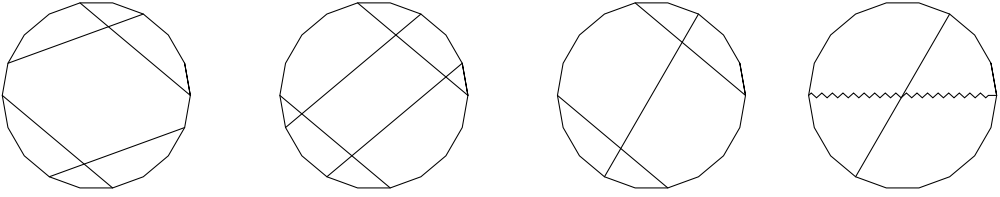
\begin{figure}
\centering{
\begin{tikzpicture}[scale=2.5,cap=round,>=latex, font=\footnotesize, font=\sansmath\sffamily
]
	
	\draw (100:0.5) -- (0:0.5);
	\draw (280:0.5) -- (180:0.5);

	\draw (60:0.5) -- (160:0.5);
	\draw (240:0.5) -- (340:0.5);
	
        	\draw (0:0.5) -- (20:0.5) -- (40:0.5) -- (60:0.5) -- (80:0.5) -- (100:0.5) -- (120:0.5)  -- (140:0.5) -- (160:0.5) -- (180:0.5) -- (200:0.5) -- (220:0.5) -- (240:0.5) -- (260:0.5) -- (280:0.5) -- (300:0.5) -- (320:0.5) -- (340:0.5) -- (360:0.5) -- (380:0.5);
  \begin{scope}[xshift=42]

	\draw (100:0.5) -- (0:0.5);
	\draw (280:0.5) -- (180:0.5);

	\draw (60:0.5) -- (200:0.5);
	\draw (240:0.5) -- (20:0.5);
	
        	\draw (0:0.5) -- (20:0.5) -- (40:0.5) -- (60:0.5) -- (80:0.5) -- (100:0.5) -- (120:0.5)  -- (140:0.5) -- (160:0.5) -- (180:0.5) -- (200:0.5) -- (220:0.5) -- (240:0.5) -- (260:0.5) -- (280:0.5) -- (300:0.5) -- (320:0.5) -- (340:0.5) -- (360:0.5) -- (380:0.5);
\end{scope}
\begin{scope}[xshift = 84]

	\draw (100:0.5) -- (0:0.5);
	\draw (280:0.5) -- (180:0.5);

	\draw (60:0.5) -- (240:0.5);
	
        	\draw (0:0.5) -- (20:0.5) -- (40:0.5) -- (60:0.5) -- (80:0.5) -- (100:0.5) -- (120:0.5)  -- (140:0.5) -- (160:0.5) -- (180:0.5) -- (200:0.5) -- (220:0.5) -- (240:0.5) -- (260:0.5) -- (280:0.5) -- (300:0.5) -- (320:0.5) -- (340:0.5) -- (360:0.5) -- (380:0.5);
\end{scope}
\begin{scope}[xshift = 122]

\draw[decorate, decoration={
    zigzag,
    segment length=4,
    amplitude=.9,post=lineto,
    post length=2pt
}] (180:0.5) -- (0:0.5);

	\draw (60:0.5) -- (240:0.5);
	
        	\draw (0:0.5) -- (20:0.5) -- (40:0.5) -- (60:0.5) -- (80:0.5) -- (100:0.5) -- (120:0.5)  -- (140:0.5) -- (160:0.5) -- (180:0.5) -- (200:0.5) -- (220:0.5) -- (240:0.5) -- (260:0.5) -- (280:0.5) -- (300:0.5) -- (320:0.5) -- (340:0.5) -- (360:0.5) -- (380:0.5);
\end{scope}
\end{tikzpicture}
}
\caption{The pictures illustrate from left to right: Two pairs of arcs crossing once, two pairs of arcs crossing twice, a diameter crossing a pair of arcs and two diameters of different colour crossing}
  \label{fig:crossing}
\end{figure}

The combinatorial model of the cluster category of Dynkin type $D_n$ obtained through the bijection $b: \AR(\CC_{D_n})_0 \to \A(\p_{2n})$ is closely related to the model by Schiffler from \cite{Schiffler}, where triangulations of the punctured disc were used to first combinatorially describe the cluster category of Dynkin type $D_n$. In particular, Proposition 1.3 in Schiffler's paper \cite{Schiffler} remains valid and can be restated as follows.

\begin{lemma} \label{intersection number}
Let $\overline{(i,j)}$ and $\overline{(k,l)}$ be diameters or pairs of arcs in $\A(\p_{2n})$ with corresponding indecomposable objects $m_{\overline{(i,j)}}$ and $m_{\overline{(k,l)}}$ in $\CC_{D_n}$. Then the number of times $\overline{(i,j)}$ and $\overline{(k,l)}$ cross is equal to the dimension $\dim \Ext_{\CC_{D_n}}^1(m_{\overline{(i,j)}},m_{\overline{(k,l)}})$ of the extension space.
\end{lemma}

\section{Torsion pairs in the cluster category of Dynkin type $D_n$} 
\label{Torsion pairs}

Torsion pairs in triangulated categories were introduced by Iyama and Yoshino in \cite{IY}. The idea follows the classical concept of torsion theory in abelian categories, which goes back to Dickson \cite{Dickson}. Torsion pairs in the cluster category of Dynkin type $D_n$ have been classified by Holm, J\o rgensen and Rubey in \cite{HJR} via the combinatorial model described in Section \ref{Combinatorial model}. First, we recall the definition of a torsion pair in a triangulated category.

\begin{definition}[\cite{IY}]
A {\em torsion pair} in a triangulated category $T$ with shift functor $\Sigma$ is a pair $(X,Y)$ of subcategories in $T$ such that
\begin{itemize}
\item[T1]{$\Hom(x,y) = 0$ for all $x \in X$ and $y \in Y$.}
\item[T2]{For each $t \in T$ there exists a distinguished triangle
$$
x \to t \to y \to \Sigma x
$$
with $x \in X$ and $y \in Y$.}
\end{itemize}
In a torsion pair $(X,Y)$, the subcategory $X$ is called the {\em torsion part} and the subcategory $Y$ is called the {\em torsion-free part} of $(X,Y)$. 
\end{definition}

Recall that every torsion pair $(X,Y)$ is uniquely determined by its torsion part $X$:
$$
Y = X^\perp := \{t \in  T| \Hom(x,t) = 0\; \forall x \in X\},
$$
or equivalently by its torsion-free part: $X = \leftidx{^\perp}Y$. By  Proposition 2.3 in Iyama and Yoshino's paper \cite{IY}, a contravariantly finite subcategory $X$ of $T$ is the torsion part of a torsion pair in $T$ if and only if $^\perp (X^\perp) = X$. Contravariantly finite means that for every object $t \in T$ there is a morphism $f: x \to t$, such that $x \in X$ and any other morphism from the subcategory $X$ into $t$ factors through $f$. Because $\CC_{D_n}$ has only finitely many indecomposable objects, every subcategory $X$ of $\CC_{D_n}$ is contravariantly finite. Thus the torsion parts of torsion pairs in $\CC_{D_n}$ are precisely those subcategories $X$ which satisfy $^\perp (X^\perp) = X$. This condition translates rather nicely to the combinatorial model: Let $X$ be a subcategory of $T$ and let $\X$ be the corresponding diagram of Dynkin type $D_n$. Set 
\begin{eqnarray*}
\nc \X = \{\overline{(i,j)} \in \A(\p_{2n})| \overline{(i,j)} \text{ crosses no element of } \X\}.
\end{eqnarray*}
By Proposition 3.5 in Holm, J\o rgensen and Rubey's paper \cite{HJR}, the pair $(X,X^\perp)$ is a torsion pair in $\CC_{D_n}$ if and only if $\X = \nc \nc \X$. In light of this result, the problem of classifying torsion pairs in $\CC_{D_n}$ boils down to finding a combinatorial description for diagrams $\X$ of Dynkin type $D_n$ satisfying $\X = \nc \nc \X$. Diagrams satisfying this condition are called Ptolemy diagrams of Dynkin type $D_n$ and can be described combinatorially by considering all crossing elements. Roughly speaking, whenever two pairs of arcs or diameters in a Ptolemy diagram cross, their convex hull has to be contained in the Ptolemy diagram as well.

\begin{definition}[Ptolemy diagram of Dynkin type $D_n$, \cite{HJR} , Definition 4.1]
Let $\X$ be a diagram of Dynkin type $D_n$. It is called a {\em Ptolemy diagram of Dynkin type $D_n$} if for any $\overline{(i,j)}, \overline{(k,l)} \in \X$, such that the arcs $(i,j)$ and $(k,l)$ cross, the following holds.
\begin{itemize}
\item[Pt1]{If $\overline{(i,j)}$ and $\overline{(k,l)}$ are pairs of arcs, then each of $\overline{(i,k)}$, $\overline{(k,j)}$, $\overline{(j,l)}$ and $\overline{(l,i)}$ lies in $\X \cup \mathcal{E}(\p_{2n})$. If any of $(i,k)$, $(k,j)$, $(j,l)$ or $(l,i)$ is a diameter, then both the red and the green copy of that diameter lie in $\X$.
}
\item[Pt2]{If $\overline{(i,j)}$ and $\overline{(k,l)}$ are diameters of different colour, then $\overline{(i,k)} = \overline{(j,l)}$ and $\overline{(k,j)} = \overline{(l,i)}$ lie in $\X \cup \mathcal{E}(\p_{2n})$.
}
\item[Pt3]{If $\overline{(i,j)}$ is a diameter and $\overline{(k,l)}$ is a pair of arcs, then those of $\overline{(i,k)}$, $\overline{(k,j)}$, $\overline{(j,l)}$ and $\overline{(l,i)}$ which do not cross $\overline{(k,l)}$ lie in $\X \cup \mathcal{E}(\p_{2n})$. Furthermore, the diameters $\overline{(k,k+n)}$ and $\overline{(l,l+n)}$ of the same colour as $\overline{(i,j)}$ also lie in $\X$.
}
\end{itemize}
Figure \ref{fig:PT} illustrates the axioms for a Ptolemy diagram.
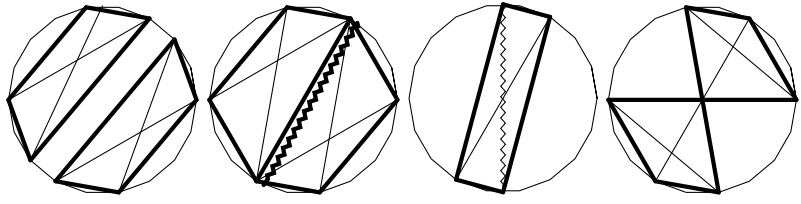
\begin{figure}
\centering{
\begin{tikzpicture}[scale=2.5,cap=round,>=latex, font = \footnotesize, font=\sansmath\sffamily
]
	
	\draw[ultra thick] (60:0.5) -- (220:0.5);
	\draw[ultra thick] (60:0.5) -- (100:0.5);
	\draw[ultra thick] (100:0.5) -- (180:0.5);
	\draw[ultra thick] (180:0.5) -- (220:0.5);

	\draw (60:0.5cm) -- (180:0.5);
	\draw (220:0.5)--(90:0.5cm);

	\draw[black, ultra thick] (240:0.5) -- (40:0.5);
	\draw[black, ultra thick] (240:0.5) -- (280:0.5);
	\draw[black, ultra thick] (280:0.5) -- (0:0.5);
	\draw[black, ultra thick] (0:0.5) -- (40:0.5);

	\draw[black] (240:0.5cm) -- (0:0.5);
	\draw[black] (40:0.5)--(280:0.5cm);

        	\draw (0:0.5) -- (20:0.5) -- (40:0.5) -- (60:0.5) -- (80:0.5) -- (100:0.5) -- (120:0.5)  -- (140:0.5) -- (160:0.5) -- (180:0.5) -- (200:0.5) -- (220:0.5) -- (240:0.5) -- (260:0.5) -- (280:0.5) -- (300:0.5) -- (320:0.5) -- (340:0.5) -- (360:0.5) -- (380:0.5);
  \end{tikzpicture}
\begin{tikzpicture}[scale=2.5,cap=round,>=latex, font=\sansmath\sffamily
]
       
	\draw[decorate, decoration={
    zigzag,
    segment length=4,
    amplitude=.9,post=lineto,
    post length=2pt
}, ultra thick] (55:0.5) -- (245:0.5);
	\draw[ultra thick] (60:0.5) -- (240:0.5);

	\draw[ultra thick] (60:0.5) -- (100:0.5);
	\draw[ultra thick] (100:0.5) -- (180:0.5);
	\draw[ultra thick] (180:0.5) -- (240:0.5);

	\draw (60:0.5cm) -- (180:0.5);
	\draw (240:0.5)--(100:0.5cm);

	\draw[black, ultra thick] (240:0.5) -- (280:0.5);
	\draw[black, ultra thick] (280:0.5) -- (0:0.5);
	\draw[black, ultra thick] (0:0.5) -- (60:0.5);

	\draw[black] (240:0.5cm) -- (0:0.5);
	\draw[black] (60:0.5)--(280:0.5cm);

        	\draw (0:0.5) -- (20:0.5) -- (40:0.5) -- (60:0.5) -- (80:0.5) -- (100:0.5) -- (120:0.5)  -- (140:0.5) -- (160:0.5) -- (180:0.5) -- (200:0.5) -- (220:0.5) -- (240:0.5) -- (260:0.5) -- (280:0.5) -- (300:0.5) -- (320:0.5) -- (340:0.5) -- (360:0.5) -- (380:0.5);
  \end{tikzpicture}
\begin{tikzpicture}[scale=2.5,cap=round,>=latex, font=\sansmath\sffamily
]
        
	\draw[ultra thick] (60:0.5) -- (90:0.5);
	\draw[ultra thick] (90:0.5) -- (240:0.5);
	\draw[ultra thick] (240:0.5) -- (270:0.5);
	\draw[ultra thick] (270:0.5) -- (60:0.5);

	\draw (60:0.5cm) -- (240:0.5);
	\draw[decorate, decoration={
    zigzag,
    segment length=4,
    amplitude=.9,post=lineto,
    post length=2pt
}] (90:0.5)--(270:0.5cm);

        	\draw (0:0.5) -- (20:0.5) -- (40:0.5) -- (60:0.5) -- (80:0.5) -- (100:0.5) -- (120:0.5)  -- (140:0.5) -- (160:0.5) -- (180:0.5) -- (200:0.5) -- (220:0.5) -- (240:0.5) -- (260:0.5) -- (280:0.5) -- (300:0.5) -- (320:0.5) -- (340:0.5) -- (360:0.5) -- (380:0.5);
  \end{tikzpicture}
\begin{tikzpicture}[scale=2.5,cap=round,>=latex, font=\sansmath\sffamily
]
	
	\draw (100:0.5) -- (0:0.5);
	\draw (60:0.5) -- (240:0.5);
	\draw[black] (280:0.5) -- (180:0.5);

	\draw[ultra thick] (60:0.5) -- (100:0.5);
	\draw[ultra thick] (0:0.5) -- (60:0.5);
	\draw[black, ultra thick] (240:0.5) -- (280:0.5);
	\draw[black, ultra thick] (180:0.5) -- (240:0.5);

	\draw[ultra thick] (100:0.5) -- (280:0.5);
	\draw[ultra thick] (0:0.5) -- (180:0.5);

        	\draw (0:0.5) -- (20:0.5) -- (40:0.5) -- (60:0.5) -- (80:0.5) -- (100:0.5) -- (120:0.5)  -- (140:0.5) -- (160:0.5) -- (180:0.5) -- (200:0.5) -- (220:0.5) -- (240:0.5) -- (260:0.5) -- (280:0.5) -- (300:0.5) -- (320:0.5) -- (340:0.5) -- (360:0.5) -- (380:0.5);
  \end{tikzpicture}
}
\caption{The axioms for a Ptolemy diagram of Dynkin type $D_n$ illustrated: Whenever two elements, drawn with thin lines, of a Ptolemy diagram $\X$ of Dynkin type $D_n$ cross, then the thick lines must be contained in $\X \cup \mathcal{E}(\p_{2n})$}
  \label{fig:PT}
\end{figure}

\end{definition}

\begin{theorem}[\cite{HJR}, Theorem 1.1] \label{classification of torsion pairs}
Let $X$ be a subcategory of $\CC_{D_n}$ and $\X$ be the corresponding diagram of Dynkin type $D_n$. Then the following are equivalent.
\begin{itemize}
\item{The pair $(X,X^{\perp})$ is a torsion pair in $\CC_{D_n}$.}
\item{The diagram $\X$ is a Ptolemy diagram of Dynkin type $D_n$.}
\end{itemize}
\end{theorem}

\begin{remark}
Every cluster tilting subcategory $X$ in a triangulated category $T$ gives rise to a torsion pair $(X,X^\perp)$ in $T$. In $\CC_{D_n}$, cluster tilting subcategories are just maximal rigid subcategories and as a consequence of Lemma \ref{intersection number} correspond to diagrams of Dynkin type $D_n$ that are maximal collections of non-crossing elements of $\A(\p_{2n})$. These diagrams trivially satisfy the Ptolemy condition, as they consist of pairwise non-crossing pairs of arcs and diameters.
\end{remark}

\section{Mutation of torsion pairs in triangulated categories}
\label{Mutation of torsion pairs}

Mutation of torsion pairs in triangulated categories has been defined by Iyama and Yoshino in \cite{IY} and in cluster categories provides a generalisation of mutation of cluster tilting subcategories. Let $T$ be a triangulated category with shift functor $\Sigma$.

\begin{definition}[mutation in triangulated categories, \cite{IY}] \label{D:mutation}
Fix a rigid subcategory $D$ of $T$, i.e.\;a subcategory $D$ of $T$ such that for all objects $d$ and $d'$ in $D$ it holds $\Ext^1_T(d,d') = 0$. For a subcategory $M \subset T$, the mutations of $M$ with respect to $D$ are the subcategories 
\begin{itemize}
\item{$\mu^-_D(M)$ of objects $t \in \leftidx{^\perp}{(\Sigma D)}$ such that there exists a distinguished triangle 
\begin{eqnarray*}\label{left}
\xymatrix{m \ar[r] & d \ar[r] & t \ar[r] & \Sigma m}
\end{eqnarray*}
with $m \in M$ and $d \in D$. 
}
\item{$\mu_D(M)$ of objects $t \in (\Sigma^{-1}D)^\perp$ such that there exists a distinguished triangle 
\begin{eqnarray*}\label{right}
\xymatrix{t \ar[r] & d \ar[r] & m \ar[r] & \Sigma t}
\end{eqnarray*}
with $m \in M$ and $d \in D$. 
}
\end{itemize}

A pair $(M,N)$ of subcategories $M,N \subset T$ is called a {\em $D$-mutation pair} if
$$
D \subset N \subset \mu^-_D(M)
\quad \text{ and } \quad
D \subset M \subset \mu_D(N).
$$
\end{definition}

\begin{lemma}\label{invariant}
For any rigid subcategory $D$ of $T$ we have $\mu_D(D) = \mu^-_D(D) = D$.
\end{lemma}

\begin{proof}
Let $d$ be any object in $D$ and let
\begin{eqnarray*}\label{triangle1}
\xymatrix{d \ar[r] & d' \ar[r] & t \ar[r]^-0 & \Sigma d},
\end{eqnarray*}
be a distinguished triangle with $t \in$ $\leftidx{^\perp}{(\Sigma D)}$. The direct sum of the distinguished triangles
$$
\xymatrix{d \ar@{=}[r] & d \ar[r] & 0 \ar[r] & \Sigma d}
$$
and
$$
\xymatrix{0 \ar[r]& t \ar@{=}[r]& t \ar[r] & 0}
$$
is again a distinguished triangle and we have an isomorphism of triangles
$$
\xymatrix{d \ar[r]^{\begin{pmatrix}1&0\end{pmatrix}} \ar@{=}[d]& d \oplus t \ar[d]^{\cong} \ar[r]^{\begin{pmatrix}0\\1\end{pmatrix}} & t \ar@{=}[d] \ar[r]^-0 & \Sigma d \ar@{=}[d]\\
d \ar[r]^-f & d' \ar[r] & t \ar[r]^-0 & \Sigma d}
$$
We get $d' \cong d \oplus t$ and because every subcategory is assumed to be closed under direct summands, we have $t \in D$. Thus we have $\mu^-_D(D) \subset D$. On the other hand for any object $\tilde{d}$ in $D$, one can complete the $D$-approximation $0 \to \tilde{d}$ of $0$ to the distinguished triangle $\xymatrix{0 \ar[r] & \tilde{d} \ar@{=}[r] & \tilde{d} \ar[r] & 0}$ and thus we have equality; $\mu^-_D(D) = D$. Dually one shows that $\mu_D(D) = D$.
\end{proof}

\begin{remark}
By Proposition 2.6 in Iyama and Yoshino's paper \cite{IY}, for any $D$-mutation pair $(M,N)$ in $T$ we have $M = \mu_D(N)$ and $N = \mu^-_D(M)$. This means that the mutations $\mu_D$ and $\mu_D^-$ are mutually inverse, i.e.\;$\mu_{D}(\mu^-_{D}(M))=M$ and $\mu^-_{D}(\mu_{D}(N))=N$.
\end{remark}

In triangulated categories with Auslander Reiten triangles, mutation of a torsion pairs has been defined by Zhou and Zhu in \cite{ZZ}. Assume now that the triangulated category $T$ has Auslander Reiten triangles, and let $\tau$ be the Auslander Reiten translation.

\begin{theorem}[\cite{ZZ}]\label{Zhou-Zhu}
Let $(X,X^\perp)$ be a torsion pair in $T$, and let $D \subset X \cap (\Sigma^{-1} X)^\perp$ be a functorially finite rigid subcategory satisfying $\tau D = \Sigma D$. Then the pairs of subcategories $\mu_D(X,X^\perp):= (\mu_D(X), \mu_{\Sigma D}(X^\perp))$ and $\mu^-_D(X,X^\perp) := (\mu^-_D(X), \mu^-_{\Sigma D}(X^\perp))$ are torsion pairs in $T$.
\end{theorem}

Mutation of a torsion pair $(X,X^\perp)$ is thus defined with respect to a subcategory $D$ of $X$, such that there are no extensions from $X$ to $T$ and that additionally satisfies the $2$-Calabi-Yau condition $\tau D = \Sigma D$. The latter is automatic for any subcategory of a $2$-Calabi-Yau category, hence we do not have to worry about it when working in the cluster category $\CC_{D_n}$ of Dynkin type $D_n$.

\section{Non-crossing diagrams of Dynkin type $D_n$ and mutation} 
\label{Non-crossing diagrams}

Zhou and Zhu introduced $\mathcal{D}$-cells for non-crossing subdiagrams $\mathcal{D}$ of Ptolemy diagrams of Dynkin type $A_n$ and $A_{\infty}$ in \cite{ZZ}. In this section, we define the analogue for Dynkin type $D_n$, which will prove essential for the definition of mutation of Ptolemy diagrams in Section \ref{A combinatorial model for mutation}. Let from now on $\D$ denote a non-crossing diagram of Dynkin type $D_n$, i.e.\;a diagram of Dynkin type $D_n$ with pairwise non-crossing elements. Informally speaking, $\D$-cells of Dynkin type $D_n$ are convex polygons with edges in $\D \cup \mathcal{E}(\p_{2n})$ which do not contain any diagonals in $\D$. However, the presence of diameters means that we have to be careful with the definition. It is useful to replace some of the diameters in $\A(\p_{2n})$ by pairs of radii. 

\begin{definition}\label{radii}
We introduce an additional central vertex $c$, which is placed at the centre of $\p_{2n}$ and additional arcs $(x,c)$ for $x \in \{0, \ldots, 2n-1\}$, which we call {\em radii}. The $\pi$-rotation $(x+n,c)$ of a radius $(x,c)$ is again a radius and together they form a pair of radii. For each pair of radii $\{(x,c),(x+n,c)\}$ we introduce a copy of colour red and a copy of colour green and denote it by $\overline{(x,c)}_r$, respectively $\overline{(x,c)}_g$. 

Let $\D$ be a non-crossing diagram. We define the replacement map $r_{\D}$ as follows. If $\D$ has no diameters, then we set $r_{\D}: \A(\p_{2n}) \cup \mathcal{E}(\p_{2n}) \to
\A(\p_{2n}) \cup \mathcal{E}(\p_{2n})$ to be the identity. Otherwise we define
\begin{eqnarray*}
r_{\D}: \A(\p_{2n}) \cup \mathcal{E}(\p_{2n}) &\to&
\A(\p_{2n}) \cup \mathcal{E}(\p_{2n})\cup \{\overline{(x,c)}| x \in \{0, \ldots, 2n-1\} \}
\end{eqnarray*}
by
\begin{eqnarray*}
\overline{(x,y)} &\mapsto& \overline{(x,y)} \text{ for all pair of arcs }\overline{(x,y)} \in \A(\p_{2n}) \cup \mathcal{E}(\p_{2n})\\
\overline{(x,x+n)}_{r,g}&\mapsto& \begin{cases}
\overline{(x,x+n)}_{r,g} \text{ if the differently coloured diameter } \overline{(x,x+n)}_{g,r} \text{ lies in } \D \\
\overline{(x,c)}_{r,g} \text{ otherwise.} 
\end{cases}
\end{eqnarray*}
\end{definition}

Figure \ref{fig:r_D} illustrates examples for the map $r_{\D}$ for a diagram of diameters with non-crossing subdiagram $\D$.

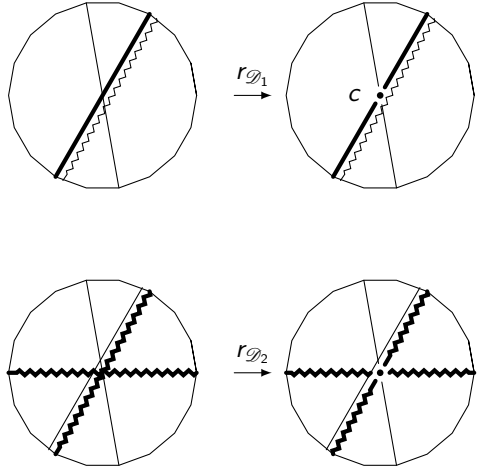
\begin{figure}
\centering{
\begin{tikzpicture}[scale=2.5,cap=round,>=latex, font = \footnotesize, font=\sansmath\sffamily
]

	\draw[decorate, decoration={
    zigzag,
    segment length=4,
    amplitude=.9,post=lineto,
    post length=2pt
}] (55:0.5) -- (245:0.5);

	\draw[ultra thick] (60:0.5cm) -- (240:0.5);
	\draw (100:0.5) -- (280:0.5);
	
        	\draw (0:0.5) -- (20:0.5) -- (40:0.5) -- (60:0.5) -- (80:0.5) -- (100:0.5) -- (120:0.5)  -- (140:0.5) -- (160:0.5) -- (180:0.5) -- (200:0.5) -- (220:0.5) -- (240:0.5) -- (260:0.5) -- (280:0.5) -- (300:0.5) -- (320:0.5) -- (340:0.5) -- (360:0.5) -- (380:0.5);
        	
        	\draw[->] (0.7,0) -- (0.9,0) node[midway, above] {$r_{\D_1}$};
 
\begin{scope}[xshift = 42]
        
           \node[label=left:{$c$}] at (0,0) {};
	\draw (100:0.5) -- (100:0.05);
	\draw (280:0.05) -- (280:0.5);
	
	\draw[decorate, decoration={
    zigzag,
    segment length=4,
    amplitude=.9,post=lineto,
    post length=2pt
}] (55:0.5) -- (245:0.5);
	\fill (0,0) circle (0.5pt);
	
	\draw[ultra thick] (60:0.5cm) -- (60:0.05);
	\draw[ultra thick] (240:0.05)--(240:0.5cm);

        	\draw (0:0.5) -- (20:0.5) -- (40:0.5) -- (60:0.5) -- (80:0.5) -- (100:0.5) -- (120:0.5)  -- (140:0.5) -- (160:0.5) -- (180:0.5) -- (200:0.5) -- (220:0.5) -- (240:0.5) -- (260:0.5) -- (280:0.5) -- (300:0.5) -- (320:0.5) -- (340:0.5) -- (360:0.5) -- (380:0.5);
        	\end{scope}

  \begin{scope}[yshift=-42]
        
        \draw (100:0.5) -- (280:0.5);
	
	\draw[ultra thick, decorate, decoration={
    zigzag,
    segment length=4,
    amplitude=.9,post=lineto,
    post length=2pt
}] (60:0.5) -- (240:0.5);

	\draw (65:0.5) -- (235:0.5);
	\draw[ultra thick, decorate, decoration={
    zigzag,
    segment length=4,
    amplitude=.9,post=lineto,
    post length=2pt
}] (0:0.5cm) -- (180:0.5);

        	\draw (0:0.5) -- (20:0.5) -- (40:0.5) -- (60:0.5) -- (80:0.5) -- (100:0.5) -- (120:0.5)  -- (140:0.5) -- (160:0.5) -- (180:0.5) -- (200:0.5) -- (220:0.5) -- (240:0.5) -- (260:0.5) -- (280:0.5) -- (300:0.5) -- (320:0.5) -- (340:0.5) -- (360:0.5) -- (380:0.5);
        	
        	\draw[->] (0.7,0) -- (0.9,0) node[midway, above] {$r_{\D_2}$};
 \end{scope}
\begin{scope}[yshift=-42, xshift = 42]
        
	\draw (100:0.5) -- (100:0.05);
	\draw (280:0.5) -- (280:0.05);

	\draw[ultra thick, decorate, decoration={
    zigzag,
    segment length=4,
    amplitude=.9,post=lineto,
    post length=2pt
}] (180:0.05) -- (180:0.5);
\draw[ultra thick, decorate, decoration={
    zigzag,
    segment length=4,
    amplitude=.9,post=lineto,
    post length=2pt
}] (0:0.05) -- (0:0.5);
	\draw[ultra thick, decorate, decoration={
    zigzag,
    segment length=4,
    amplitude=.9,post=lineto,
    post length=2pt
}] (60:0.5) -- (60:0.05);
\draw[ultra thick, decorate, decoration={
    zigzag,
    segment length=4,
    amplitude=.9,post=lineto,
    post length=2pt
}] (240:0.5) -- (240:0.05);
\draw (65:0.5) -- (235:0.5);
	\fill (0,0) circle (0.5pt);
	
        	\draw (0:0.5) -- (20:0.5) -- (40:0.5) -- (60:0.5) -- (80:0.5) -- (100:0.5) -- (120:0.5)  -- (140:0.5) -- (160:0.5) -- (180:0.5) -- (200:0.5) -- (220:0.5) -- (240:0.5) -- (260:0.5) -- (280:0.5) -- (300:0.5) -- (320:0.5) -- (340:0.5) -- (360:0.5) -- (380:0.5);
        	\end{scope}
  \end{tikzpicture}
}
  \caption{The diameters lying in the non-crossing subdiagrams $\D_1$ respectively $\D_2$ are marked by thick lines. The picture shows how some diameters in $\A(\p_{2n})$ get replaced by pairs of radii and some stay as diameters under the replacement map $r_{\D_1}$ respectively $r_{\D_2}$}
  \label{fig:r_D}
  \end{figure}

\begin{definition}
A pair of radii $\overline{(x,c)}_{r,g}$ is said to cross a diameter or pair of arcs $\overline{(a,b)}$ if and only if the corresponding diameter $\overline{(x,x+n)}_{r,g}$ crosses $\overline{(a,b)}$. Two pairs of radii $\overline{(x,c)}_{r,g}$ and $\overline{(a,c)}_{g,r}$ are said to cross if and only if the corresponding diameters $\overline{(x,x+n)}_{r,g}$ and $\overline{(a,a+n)}_{g,r}$ cross. Furthermore, two radii $(a,c)$ and $(x,c)$ do not cross for any $a,x \in \{0, \ldots, 2n-1\}$. 
\end{definition}

\begin{lemma}\label{non-crossing}
If $\overline{(a,b)}$ and $\overline{(x,y)}$ lie in $r_{\D}(\D)$ then the arcs $(a,b)$ and $(x,y)$ do not cross.
\end{lemma}
\begin{proof}
Assume by contradiction that the arcs $(a,b)$ and $(x,y)$ cross and $\overline{(a,b)}$ and $\overline{(x,y)}$ lie in $r_{\D}(\D)$. Because $\D$ is non-crossing, the elements $\overline{(a,b)}$ and $\overline{(x,y)}$ must be diameters of the same colour in $\D$. However, if $\D$ contains two distinct diameters of the same colour, they get replaced by pairs of radii when applying the replacement map $r_{\D}$ which contradicts the assumption that $\overline{(a,b)}$ and $\overline{(x,y)}$ lie in $r_{\D}(\D)$.
\end{proof}

\begin{definition}
For two arcs $(x,y)$ and $(y,z)$ we denote by $\sphericalangle(x,y,z)$ the angle covered when rotating $(x,y)$ to $(y,z)$ in a clockwise direction. We assume $0 \leq \sphericalangle(x,y,z) < 2 \pi$.
\end{definition}

\begin{lemma}\label{zeroangle}
Let $\overline{(x,y)}, \overline{(y,z)},\overline{(y,z')} \in r_{\D}(\D)$. Then we have $\sphericalangle(x,y,z) = \sphericalangle(x,y,z')$ if and only if $z = z'$. 
\end{lemma}
\begin{proof}
Assume that $\sphericalangle(x,y,z) = \sphericalangle(x,y,z')$. If both $z$ and $z'$ lie in $\{0, \ldots, 2n-1\}$ it follows from the regularity of $\p_{2n}$ that $z = z'$. Otherwise, if $c \in \{z,z'\}$ then $z \neq z'$ would imply $\{z,z'\} = \{c,y+n\}$. However, in this case both $\overline{(y,y+n)}_r$ or $\overline{(y,y+n)}_g$ and $\overline{(y,c)}_r$ or $\overline{(y,c)}_g$ are in $r_{\D}(\D)$, which contradicts the definition of the map $r_{\D}$.
\end{proof}

\begin{definition}[$\D$-cell of Dynkin type $D_n$]\label{D-cells}
Let $\D$ be a non-crossing diagram of Dynkin type $D_n$. A $\D${\em -cell of Dynkin type} $D_n$  is a polygon $\langle d_1, \ldots, d_k \rangle$ with consecutive vertices $d_1, \ldots, d_k \in \{0, \ldots, 2n-1\} \cup \{c\}$, $k \geq 3$, such that for $i = 0, \ldots, k-1$ we have
$$
\overline{(d_i,d_{i+1})} \in r_{\D}(\D) \cup \mathcal{E}(\p_{2n}),
$$
where we calculate modulo $k$ in the indices.
Furthermore, for any $\overline{(d_i,v)} \in r_{\D}(\D)$ with $\sphericalangle(d_{i-1}, d_i, v) >0$ we have
$$
0 < \sphericalangle(d_{i-1}, d_i, d_{i+1}) \leq \sphericalangle(d_{i-1}, d_i, v).
$$

We call a $\D$-cell $\langle d_1, \ldots, d_k \rangle$ of Dynkin type $D_n$ together with its $\pi$-rotation $\langle d_1+n, \ldots, d_k+n \rangle$ (where we set $c+n = c$) a {\em pair of $\D$-cells} and denote it by $\overline{\langle d_1, \ldots, d_k\rangle}$. We call a pair of $\D$-cells {\em central}, if they contain the centre $c$ of the polygon $\p_{2n}$.
\end{definition}

\begin{example}

Figure \ref{fig:C1} shows an example of a pair of $\D$-cells which does not contain the centre of $\p_{2n}$. Figures \ref{fig:C2} shows examples of central pairs of $\D$-cells.

\begin{figure}
\centering{
\begin{tikzpicture}[scale=2.5,cap=round,>=latex, font = \footnotesize, font=\sansmath\sffamily
]
        
	\fill[black!20] (0:0.5) -- (20:0.5) -- (80:0.5) -- (120:0.5) -- (0:0.5) -- cycle;
	\fill[black!20] (180:0.5) -- (200:0.5) -- (80+180:0.5) -- (120+180:0.5) -- (180:0.5) -- cycle;
	\node (d1) at (0:0.57cm){$d_1$};
	\node (d2) at (20:0.57cm){$d_2$};
	\node (d3) at (80:0.57cm){$d_3$};
	\node (dk) at (120:0.57cm){$d_k$};
	
	\draw (0:0.5) -- (20:0.5)--(80:0.5);
	\draw[dotted] (80:0.5) -- (120:0.5);
	\draw (120:0.5) -- (0:0.5);
	
	\draw (180:0.5) -- (200:0.5)--(250:0.5);
	\draw[dotted] (260:0.5) -- (300:0.5);
	\draw (300:0.5) -- (180:0.5);

        	\draw (0:0.5) -- (20:0.5) -- (40:0.5) -- (60:0.5) -- (80:0.5) -- (100:0.5) -- (120:0.5)  -- (140:0.5) -- (160:0.5) -- (180:0.5) -- (200:0.5) -- (220:0.5) -- (240:0.5) -- (260:0.5) -- (280:0.5) -- (300:0.5) -- (320:0.5) -- (340:0.5) -- (360:0.5) -- (380:0.5);
\end{tikzpicture}
     \caption{Example for a non-central pair of $\D$-cells} \label{fig:C1}
}
\end{figure}
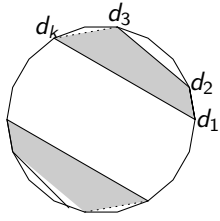

\begin{figure}
\centering
{ 
\begin{tikzpicture}[scale=2.5,cap=round,>=latex, font = \footnotesize, font=\sansmath\sffamily
]
	\fill[black!20] (60:0.5) -- (120:0.5) --(200:0.5) --(240:0.5) -- (300:0.5) -- (380:0.5) -- (60:0.5) -- cycle;
	\node (d1) at (60:0.57cm){$d_1$};
	\node (d2) at (120:0.57cm){$d_2$};
	\node (dk) at (200:0.57cm){$d_{\frac{k}{2}}$};
	\node (dk+1) at (240:0.57cm){$d_{\frac{k}{2}+1}$};
	\node (dk+n) at (380:0.57cm){$d_{k}$};

	\draw (60:0.5) -- (120:0.5);
	\draw[dotted] (120:0.5) -- (200:0.5);
	\draw (200:0.5) -- (240:0.5) -- (300:0.5);
	\draw[dotted] (300:0.5) -- (380:0.5);
 	\draw (380:0.5)-- (60:0.5);

        	\draw (0:0.5) -- (20:0.5) -- (40:0.5) -- (60:0.5) -- (80:0.5) -- (100:0.5) -- (120:0.5)  -- (140:0.5) -- (160:0.5) -- (180:0.5) -- (200:0.5) -- (220:0.5) -- (240:0.5) -- (260:0.5) -- (280:0.5) -- (300:0.5) -- (320:0.5) -- (340:0.5) -- (360:0.5) -- (380:0.5);

\begin{scope}[xshift=42]
        
	\fill[black!20] (60:0.5) -- (120:0.5) -- (200:0.5) -- (240:0.5) -- (60+180:0.5) -- (120+180:0.5) -- (200+180:0.5) (240+180:0.5) -- (60:0.5) -- cycle;
	\node (d1') at (60:0.57cm){$d_1$};
	\node (d2') at (120:0.57cm){$d_2$};
	\node (dk+1') at (240:0.57cm){$d_{k-1}$};
	\node (dk+1') at (0,0){$\bullet$};
	\node (dk+2) at (0.1,-0.05){$d_k$};
	\fill (0,0) circle (0.5pt);
	
	\draw (60:0.5) -- (120:0.5);
	\draw[dotted] (120:0.5) -- (200:0.5);
	\draw (200:0.5) -- (240:0.5);
	\draw (60:0.5) -- (240:0.5);

	\draw (240:0.5) -- (300:0.5);
	\draw[dotted] (300:0.5) -- (20:0.5);
	\draw (380:0.5) -- (60:0.5);
	\draw (240:0.5) -- (60:0.5);

        	\draw (0:0.5) -- (20:0.5) -- (40:0.5) -- (60:0.5) -- (80:0.5) -- (100:0.5) -- (120:0.5)  -- (140:0.5) -- (160:0.5) -- (180:0.5) -- (200:0.5) -- (220:0.5) -- (240:0.5) -- (260:0.5) -- (280:0.5) -- (300:0.5) -- (320:0.5) -- (340:0.5) -- (360:0.5) -- (380:0.5);
	\end{scope}
\begin{scope}[xshift=84]
        	
	\fill[black!20] (60:0.5) -- (100:0.5) --(180:0.5) -- (0,0) -- (60:0.5) -- cycle;
	\fill[black!20] (60+180:0.5) -- (100+180:0.5) --(180+180:0.5) -- (0,0) -- (60+180:0.5) -- cycle;
	\node (d1) at (60:0.57cm){$d_1$};
	\node (d2) at (100:0.57cm){$d_2$};
	\node(dk+1) at (0,0){$\bullet$};
	\node (dk+2) at (0.1,-0.05){$d_k$};
	\fill(0,0) circle (0.5pt);
	
	\draw (60:0.5) -- (100:0.5);
	\draw[dotted] (100:0.5) -- (180:0.5);
	\draw (180:0.5) -- (0,0);
	\draw(60:0.5) -- (0,0);

	\fill(0,0) circle (0.5pt);
	
	\draw (240:0.5) -- (280:0.5);
	\draw[dotted] (280:0.5) -- (360:0.5);
	\draw (360:0.5) -- (0,0);
	\draw(240:0.5) -- (0,0);
	
        	\draw (0:0.5) -- (20:0.5) -- (40:0.5) -- (60:0.5) -- (80:0.5) -- (100:0.5) -- (120:0.5)  -- (140:0.5) -- (160:0.5) -- (180:0.5) -- (200:0.5) -- (220:0.5) -- (240:0.5) -- (260:0.5) -- (280:0.5) -- (300:0.5) -- (320:0.5) -- (340:0.5) -- (360:0.5) -- (380:0.5);
     \end{scope}
\begin{scope}[xshift=126]

	\fill[black!20] (60:0.5) -- (120:0.5) -- (200:0.5) -- (240:0.5) -- (60+180:0.5) -- (120+180:0.5) -- (200+180:0.5) (240+180:0.5) -- (60:0.5) -- cycle;
	\node (d1) at (60:0.57cm){$d_1$};
	\node (d2) at (120:0.57cm){$d_2$};
	\node (dk+1) at (240:0.57cm){$d_{k}$};
	
	\draw (60:0.5) -- (120:0.5);
	\draw[dotted] (120:0.5) -- (200:0.5);
	\draw (200:0.5) -- (240:0.5);
	\draw (60:0.5) -- (240:0.5);
	\draw[decorate, decoration={
    zigzag,
    segment length=4,
    amplitude=.9,post=lineto,
    post length=2pt
}] (57:0.5) -- (243:0.5);

	\draw (240:0.5) -- (300:0.5);
	\draw[dotted] (300:0.5) -- (20:0.5);
	\draw (20:0.5) -- (60:0.5);

        	\draw (0,0) circle(0.5cm);
     
  \end{scope}
  \end{tikzpicture}
\caption{Examples for central pairs of $\D$-cells. From left to right: When $\D$ contains no diameters; when $\D$ contains one diameter; when $\D$ contains more than one diameter and all are of the same colour; when $\D$ contains two diameters of different colour} \label{fig:C2}
}
\end{figure}
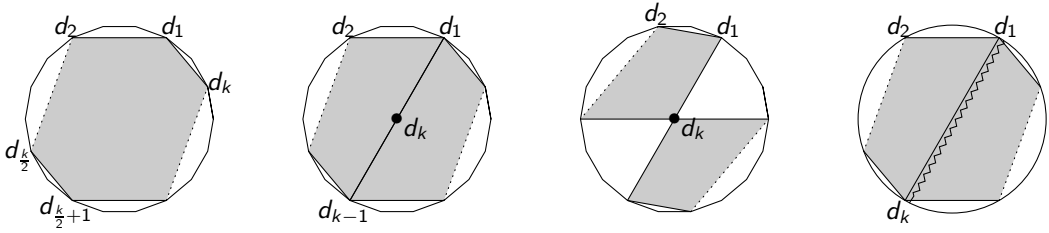

\end{example}

\begin{lemma}\label{convex}
Any $\D$-cell of Dynkin type $D_n$ is convex.
\end{lemma}

\begin{proof}
Let $\langle d_1, \ldots, d_k \rangle$ be a $\D$-cell of Dynkin type $D_n$. The interior angles are the angles  $\sphericalangle(d_{i-1}, d_i, d_{i+1})$ for $i \in \{1, \ldots, k\}$. If $d_i \in \{0, \ldots, 2n-1\}$, then $\sphericalangle(d_{i-1}, d_i, d_{i+1}) \leq \sphericalangle(d_{i-1}, d_i, d_i+1) < \pi$. If $d_i = c$, then $\sphericalangle(d_{i-1}, d_i, d_{i+1}) \leq \sphericalangle(d_{i-1}, d_i, d_{i-1}+n) = \pi$.
\end{proof}

\begin{definition}
We say that a diameter or a pair of arcs $\overline{(a,b)} \in \A(\p_{2n})$ is {\em contained in a pair of $\D$-cells $\overline{\langle d_1, \ldots, d_k \rangle}$} if $r_{\D}(\overline{(a,b)}) = \overline{(d_i,d_j)}$ for some $i,j \in \{1, \ldots, k\}$ with $j \notin \{i-1,i,i+1\}$.
\end{definition}

For a pair of arcs $\overline{(i,j)}$ to be contained in a pair of $\D$-cells $\overline{\langle d_1, \ldots, d_k \rangle}$ means that either the arc $(i,j)$ is a diagonal in $\langle d_1, \ldots, d_k \rangle$ and $(i+n,j+n)$ is a diagonal in $<d_1+n, \ldots, d_k+n>$ or vice versa. For a coloured diameter $\overline{(i,i+n)}$ to be contained in a pair of $\D$-cells $\overline{\langle d_1, \ldots, d_k \rangle}$ can mean any of the following two: Either $r_\D(\overline{(i,i+n)})$ is a diameter and it is a diagonal in $\langle d_1, \ldots, d_k \rangle$ and $\langle d_1+n, \ldots, d_k+n \rangle$ or it is a pair of radii and $\overline{(i,c)}$ is a diagonal in $\langle d_1, \ldots, d_k \rangle$ and $(i+n, c)$ is a diagonal in $\langle d_1+n, \ldots, d_k+n \rangle$ or vice versa. 

\begin{lemma}\label{divides}\label{diagonal}
Every element $\overline{(a,b)}\in \A(\p_{2n})$ which is contained in a pair of $\D$-cells $\overline{\langle d_1, \ldots, d_k\rangle}$ lies in $\nc \D \setminus \D$. On the other hand, every diameter or pair of arcs $\overline{(a,b)}\in \nc \D \setminus \D$ is contained in a unique pair of $\D$-cells.
\end{lemma}

\begin{proof}
Assume that $\overline{(a,b)} \in \A(\p_{2n})$ is contained in a pair of $\D$-cells  $\overline{\langle d_1, \ldots, d_k \rangle}$. Assume by contradiction that $\overline{(a,b)}$ lies in $\D$. Then for some $1 \leq i ,j \leq k$ with $j \notin \{i-1,i,i+1\}$, we have $r_{\D}(\overline{(a,b)}) = \overline{(d_{i},d_{j})} \in r_{\D}(\D)$, which contradicts the minimality of $\sphericalangle(d_{i-1}, d_{i}, d_{i+1})$. So we have $\overline{(a,b)} \notin \D$. Assume now that $\overline{(x,y)} \in \A(\p_{2n})$ crosses $\overline{(a,b)}$. Then either it crosses a pair of sides of $\overline{\langle d_1, \ldots, d_k \rangle}$ or it is also contained in $\overline{\langle d_1, \ldots, d_k \rangle}$. In particular, the diameter or pair of arcs $\overline{(x,y)}$ cannot lie in $\D$ and thus $\overline{(a,b)} \in \nc \D \setminus \D$. This proves the first statement of Lemma \ref{divides}.

Let $\overline{(a,b)}\in \nc \D \setminus \D$ with $r_{\D}(\overline{(a,b)}) = \overline{(a',b')}$. We first show the existence of a $\D$-cell of Dynkin type $D_n$ containing  $\overline{(a,b)}$. Construct a sequence of vertices by setting
\begin{eqnarray*}
d_1 &=& a' \\ d_2 &=& \min_{\sphericalangle(b',a',v)>0}\{v \in \{0, \ldots, 2n-1\}\cup \{c\}|\overline{(a',v)} \in r_{\D}(\D) \cup \mathcal{E}(\p_{2n})\}
\end{eqnarray*}
and for $i \geq 2$: 
\begin{eqnarray*}
d_{i+1} &=& \min_{\sphericalangle(d_{i-1},d_i,v)>0}\{v \in \{0, \ldots, 2n-1\}\cup \{c\}|\overline{(d_i,v)} \in r_{\D}(\D) \cup \mathcal{E}(\p_{2n})\}.
\end{eqnarray*}
We show that there exists a $k \in \mathbb{Z}$, such that $\langle d_1, \ldots, d_k \rangle$ is a polygon. Since the set of vertices $\{0, \ldots, 2n-1\} \cup \{c\}$ is finite, there is a $k \in \mathbb{Z}_{\geq 3}$, such that $d_{k+1} = d_i$ for some $i < k$. Chose $k$ to be minimal with this property, i.e.\;let $k$ be such that there is a $1 \leq i < k$ with $d_{k+1} = d_i$ and such that $d_1, \ldots, d_k$ are pairwise distinct.
Since by Lemma \ref{non-crossing} the arcs $(d_j,d_{j+1})$ are pairwise non-crossing, $\langle d_i, \ldots, d_k \rangle$ is a polygon. Assume, for a contradiction, that $i \geq 2$. Then $\sphericalangle(d_k,d_i,d_{i+1})>0$. Otherwise, by Lemma \ref{zeroangle} we get $d_k=d_{i+1}$ and thus $d_i=d_{i+2}$, which contradicts the condition $\sphericalangle(d_i,d_{i+1},d_{i+2})>0$. So we have
$$
\sphericalangle(d_{i-1}, d_i, d_{i+1}) \leq \sphericalangle(d_k, d_i, d_{i+1})
$$
and thus, since $(d_{i-1}, d_i)$ does not intersect any side of the polygon $\langle d_i, \ldots, d_k \rangle$, it is a diagonal or a side in $\langle d_i, \ldots, d_k \rangle$. Therefore we have $d_{i-1} = d_l$ for some $i < l \leq k$. This contradicts the assumption that $d_1, \ldots, d_k$ are pairwise distinct. Therefore $d_i = d_1$ and $\langle d_1, \ldots, d_k \rangle$ is a polygon.

By definition, for $1 < i \leq k$ the angles $\sphericalangle(d_{i-1}, d_i, d_{i+1})$ satisfy the minimality condition for angles in a $\D$-cell of Dynkin type $D_n$. Furthermore, if $\overline{(d_1,v)} \in r_{\D}(\A(\p_{2n}))$ is such that $0 < \sphericalangle(d_k,d_1,v) < \sphericalangle(d_k,d_1,d_2)$, then either it is contained in $\overline{\langle d_1, \ldots, d_k \rangle}$ or it intersects one of its pairs of sides. Thus by the first part of the proof it cannot be an element of $\D$. So $\sphericalangle(d_k,d_1,d_2)$ satisfies the minimality condition for angles in a $\D$-cell of Dynkin type $D_n$ and the polygon $\langle d_1, \ldots, d_k \rangle$ is a $\D$-cell of Dynkin type $D_n$. 

Because $r_\D(\overline{(a,b)}) = \overline{(a',b')}$ does not intersect any of the pairs of sides $\overline{(d_i,d_{i+1})}$ of $\overline{\langle d_1, \ldots, d_k\rangle}$ and because by definition of $d_2$ we have $\sphericalangle(b',a', d_2) \leq \sphericalangle(d_k,d_1,d_2)$, the arc $(a',b')$ is a diagonal in the polygon $\langle d_1, \ldots, d_k \rangle$  and thus $\overline{(a,b)}$ is contained in $\overline{\langle d_1, \ldots, d_k\rangle}$.

It remains to show uniqueness. For $\overline{(a,b)} \in \A(\p_{2n})$ with $r_{\D}(\overline{(a,b)})=\overline{(a',b')}$ let $\langle d_1, \ldots, d_k\rangle$ and $\langle d'_1, \ldots, d'_l\rangle$ be $\D$-cells of Dynkin type $D_n$ containing the arc $(a',b')$ as a diagonal. Therefore, the two convex polygons $\langle d_1, \ldots, d_k \rangle$ and $\langle d'_1, \ldots, d'_k \rangle$ share a diagonal. By Lemma \ref{diagonal}, none of the edges of $\langle d'_1, \ldots, d'_k \rangle$ can be a diagonal in  $\langle d_1, \ldots, d_k \rangle$ nor vice versa. So the two polygons share at least two edges. Since they are convex, if they are not identical then they share consecutive edges with angle $\pi$ in between, which must be a pair of radii $\overline{(x,c)} \in r_{\D}(\D)$ for some $x \in \{0, \ldots, 2n-1\}$. Then by rotational invariance of $\D$, we get $\langle d'_1, \ldots, d'_l \rangle = \langle d_1+n, \ldots, d_k+n \rangle$ and the pairs of $\D$-cells $\overline{\langle d_1, \ldots, d_k \rangle}$ and $\overline{\langle d'_1, \ldots, d'_l \rangle}$ are identical.
\end{proof}

\begin{example}
Figure \ref{fig:E} shows, for different non-crossing diagrams $\D$ of Dynkin type $D_n$, examples of elements of $\nc \D \setminus \D$ and the pairs of $\D$-cells they are contained in.

\begin{figure}
\begin{centering}
\begin{tikzpicture}[scale=2.5,cap=round,>=latex, font = \footnotesize, font=\sansmath\sffamily
]
        
	\fill[black!20] (60:0.5) -- (200:0.5) -- (160:0.5) -- (140:0.5) -- (60:0.5) -- cycle;
	\fill[black!20] (60+180:0.5) -- (200+180:0.5) -- (160+180:0.5) -- (140+180:0.5) -- (60+180:0.5) -- cycle;

	\node (d1) at (60:0.57){$a$};
	\node (d2) at (160:0.57cm){$b$};
%
	
	\draw[ultra thick] (60:0.5) -- (120:0.5);
	\draw[ultra thick] (60:0.5) -- (140:0.5);
	\draw (60:0.5) -- (160:0.5);
	\draw[ultra thick] (200:0.5) -- (240:0.5);
	\draw[ultra thick] (160+180:0.5) -- (20:0.5);
	\draw[ultra thick] (200:0.5) -- (160:0.5);
	
	\draw[ultra thick] (60:0.5) -- (200:0.5);
	\draw[ultra thick](60+180:0.5) -- (200+180:0.5);	
	
	\draw[ultra thick] (240:0.5) -- (300:0.5);
	\draw[ultra thick] (20:0.5) -- (60:0.5);
	\draw (60+180:0.5) -- (160+180:0.5);
	\draw[ultra thick] (60+180:0.5) -- (140+180:0.5);

        	\draw (0:0.5) -- (20:0.5) -- (40:0.5) -- (60:0.5) -- (80:0.5) -- (100:0.5) -- (120:0.5)  -- (140:0.5) -- (160:0.5) -- (180:0.5) -- (200:0.5) -- (220:0.5) -- (240:0.5) -- (260:0.5) -- (280:0.5) -- (300:0.5) -- (320:0.5) -- (340:0.5) -- (360:0.5) -- (380:0.5);
  
\begin{scope}[xshift = 42]
        
	\fill[black!20] (60:0.5) -- (160:0.5) -- (200:0.5) -- (240:0.5) -- (160+180:0.5) -- (20:0.5) -- (60:0.5) -- cycle;	
	\node  (d1) at (60:0.57cm){$a$};
	\node (dk+1) at (240:0.57cm){$b$};
	\fill (0,0) circle (0.5pt);
	
	\draw[ultra thick] (60:0.5) -- (120:0.5);
	\draw[ultra thick] (200:0.5) -- (240:0.5);
	\draw[ultra thick] (60:0.5) -- (240:0.5);
	\draw[decorate, decoration={
    zigzag,
    segment length=4,
    amplitude=.9,post=lineto,
    post length=2pt
}] (55:0.5) -- (245:0.5);
	\draw[ultra thick] (60:0.5) -- (160:0.5);
	\draw[ultra thick] (60+180:0.5) -- (160+180:0.5);
	\draw[ultra thick] (200:0.5) -- (240:0.5);
	\draw[ultra thick] (160:0.5) -- (200:0.5);
	\draw[ultra thick] (160+180:0.5) -- (200+180:0.5);
	
	\draw[ultra thick] (240:0.5) -- (300:0.5);
	\draw[ultra thick] (20:0.5) -- (60:0.5);

        	\draw (0:0.5) -- (20:0.5) -- (40:0.5) -- (60:0.5) -- (80:0.5) -- (100:0.5) -- (120:0.5)  -- (140:0.5) -- (160:0.5) -- (180:0.5) -- (200:0.5) -- (220:0.5) -- (240:0.5) -- (260:0.5) -- (280:0.5) -- (300:0.5) -- (320:0.5) -- (340:0.5) -- (360:0.5) -- (380:0.5);
 \end{scope}
\begin{scope}[xshift = 84]
        
	\fill[black!20] (60:0.5) -- (140:0.5) -- (160:0.5) -- (0,0) -- (20:0.5) -- (60:0.5) -- cycle;	
	\fill[black!20] (60+180:0.5) -- (140+180:0.5) -- (160+180:0.5) -- (0,0) -- (200:0.5) -- (60+180:0.5) -- cycle;	

	\node (d1) at (60:0.57cm){$a$};
	\fill (0,0) circle (0.5pt);
	\node (d1) at (240:0.57cm){$b$};

	\draw[ultra thick] (60:0.5) -- (120:0.5);
	\draw[ultra thick] (60:0.5) -- (140:0.5);
	\draw[ultra thick] (200:0.5) -- (240:0.5);
	\draw[ultra thick] (160+180:0.5) -- (20:0.5);
	\draw[ultra thick] (200:0.5) -- (160:0.5);
	\draw[ultra thick] (20:0.5) -- (200:0.5);
	\draw[ultra thick] (160+180:0.5) -- (160:0.5);
	\draw (60:0.5) -- (240:0.5); 
	\draw[ultra thick] (140:0.5) -- (160:0.5);
	\draw[ultra thick] (140+180:0.5) -- (160+180:0.5);

	\draw[ultra thick] (240:0.5) -- (300:0.5);
	\draw[ultra thick] (20:0.5) -- (60:0.5);
	\draw[ultra thick] (60+180:0.5) -- (140+180:0.5);

        	\draw (0:0.5) -- (20:0.5) -- (40:0.5) -- (60:0.5) -- (80:0.5) -- (100:0.5) -- (120:0.5)  -- (140:0.5) -- (160:0.5) -- (180:0.5) -- (200:0.5) -- (220:0.5) -- (240:0.5) -- (260:0.5) -- (280:0.5) -- (300:0.5) -- (320:0.5) -- (340:0.5) -- (360:0.5) -- (380:0.5);
\end{scope}
\end{tikzpicture}
\caption{The pictures illustrate, for different non-crossing diagrams $\D$, elements of $\nc \D \setminus \D$ and the pair of $\D$-cells they are contained in. The elements of $\D$ are marked by thick lines and the pair of $\D$-cells containing $\overline{(a,b)} \in \nc \D \setminus \D$ are marked in grey}
\label{fig:E}
\end{centering}
\end{figure}
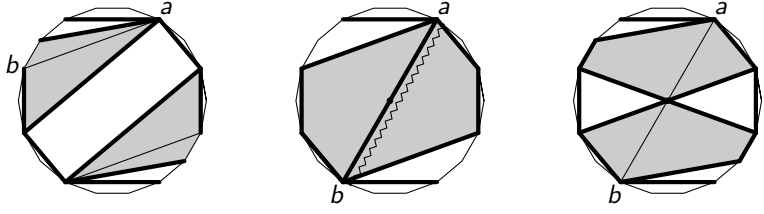
\end{example}

\begin{lemma} \label{diameters}
There exists a $\pi$-rotation invariant $\D$-cell $\langle d_1, \ldots, d_k\rangle = \langle d_1+n, \ldots, d_k+n \rangle$ of Dynkin type $D_n$ if and only if $\D$ contains no diameters. In this case, the $\pi$-rotation invariant $\D$-cell of Dynkin type $D_n$ is unique, central and it contains all diameters in $\nc \D \setminus \D$. Furthermore, if $\langle d_1, \ldots, d_k\rangle$ is a $\pi$-rotation invariant $\D$-cell of Dynkin type $D_n$, then $(d_i,d_j)$ is a diameter if and only if $(d_{i-1}, d_{j-1})$ is a diameter.
\end{lemma}

\begin{proof}
Let $\langle d_1, \ldots, d_k\rangle$ be a $\pi$-rotation invariant $\D$-cell of Dynkin type $D_n$. It is of the form
$$\langle d_1, \ldots, d_k\rangle = \langle d_1, \ldots, d_{\frac{k}{2}}, d_{\frac{k}{2}+1} = d_1 +n, \ldots, d_k=d_{\frac{k}{2}}+n \rangle,$$
with $k \geq 4$ even. In particular, $\overline{(d_i,d_j)}$ is a diameter if and only if $j = \frac{k}{2}+i$, which is the case if and only if $\overline{(d_{i-1},d_{j-1})}$ is a diameter.
Furthermore with the diameters $\overline{(d_i, d_i+n)} = \overline{(d_i,d_{\frac{k}{2}+i})}$, the $\D$-cell $\langle d_1, \ldots, d_k\rangle$ of Dynkin type $D_n$ contains the central vertex $c$. It is thus a central $\D$-cell of Dynkin type $D_n$ and $c$ lies in its interior. Any diameter $\overline{(a,a+n)} \in \A(\p_{2n})$ contains the vertex $c$ and is therefore either contained in $\langle d_1, \ldots, d_k \rangle$ or crosses one of the pairs of arcs $\overline{(d_i,d_{i+1})}$. So any $\pi$-rotation invariant $\D$-cell of Dynkin type $D_n$ contains all diameters in $\nc \D \setminus \D$ and if there is a $\pi$-rotation invariant $\D$-cell $\langle d_1, \ldots, d_k \rangle$ of Dynkin type $D_n$, there is at least one diameter in $\nc \D \setminus \D$, e.g.\;$(d_1,d_1+n)$. By Lemma \ref{divides} the $\pi$-rotation invariant $\D$-cell of Dynkin type $D_n$ is thus unique if it exists and in this case, also by Lemma \ref{diagonal}, $\D$ contains no diameters.

Suppose, on the other hand, that $\D$ contains no diameters. By Lemma 5.1 in Holm, J\o rgensen and Rubey's paper \cite{HJR}, there exists a diameter $\overline{(a,a+n)}$ in $\nc \D \setminus \D$. By rotation invariance of $\D$ and because $\D$ contains no diameters, the diameter $\overline{(a,a+n)}$ is contained in a rotation invariant $\D$-cell.
\end{proof}

We now define mutation with respect to $\D$. Informally speaking, one can think of the mutation $\mu_{\D}$, respectively $\mu^-_{\D}$ as rotating the interior of each $\D$-cell of Dynkin type $D_n$ in an anti clockwise, respectively clockwise, direction.

\begin{definition}[mutation of diagrams of Dynkin type $D_n$] \label{combinatorial mutation}
For every non-crossing diagram $\D$ of Dynkin type $D_n$ we define the mutation maps
$$
\mu_{\D}: \nc \D \to \nc \D
\quad
\text{ and }
\quad
\mu^-_{\D}: \nc \D \to \nc \D
$$
as follows.
\begin{itemize}
\item{The maps $\mu_{\D}$ and $\mu^-_{\D}$ leave $\D \subset \nc \D$ invariant:
$$
\mu_{\D}\bigr|_{\D} = \mu^-_{\D}\bigr|_{\D} = \id_{\D}.
$$
}
\item{Suppose $\overline{(a,b)} \in \nc \D \setminus \D$. By Lemma \ref{divides}, the element $\overline{(a,b)}$ is contained in a unique pair of $\D$-cells $\overline{\langle d_1, \ldots, d_k\rangle}$ and thus $r_{\D}(\overline{(a,b)}) = \overline{(d_i, d_j)}$ for some $i,j \in \{1, \ldots, k\}$. We set
$$
\mu_{\D}(\overline{(a,b)}) = r^{-1}_{\D}(\overline{(d_{i+1},d_{j+1})})
$$
and
$$
\mu^-_{\D}(\overline{(a,b)}) = r^{-1}_{\D}(\overline{(d_{i-1},d_{j-1})}),
$$
where the colour of $\mu_{\D}(\overline{(a,b)})$, respectively $\mu^-_{\D}(\overline{(a,b)})$, if it is a diameter or pair of radii, is specified as follows.
\begin{itemize}
\item{If $\D$ contains no diameters, then by Lemma \ref{diameters} all diameters are contained in the unique central pair of $\D$-cells $\langle d_1, \ldots, d_n\rangle$ and only diameters get mutated to diameters. We define both $\mu_{\D}$ and $\mu^-_{\D}$ to change their colour. So if $\overline{(a,b)}$ is a red diameter, both $\mu_{\D}(\overline{(a,b)})$ and $\mu^-_{\D}(\overline{(a,b)})$ are set to be green and vice versa.
} 
\item{If $\D$ contains more than one diameter of the same colour, then all diameters in $\D$ are of the same colour. Those of $\mu_{\D}(\overline{(a,b)})$ and $\mu^-_{\D}(\overline{(a,b)})$ which are diameters are set to be of the same colour as all the diameters in $\D$.}
\item{If $\D$ contains exactly one diameter $\overline{(x,x+n)}$, then if $\mu_{\D}(\overline{(a,b)}) = \overline{(a',b')}$, respectively $\mu^-_{\D}(\overline{(a,b)}) = \overline{(a'',b'')}$, is a diameter, it is set to be of different colour than  $\overline{(x,x+n)}$ if and only if $(a',b') = (x,x+n)$, respectively $(a'',b'') = (x,x+n)$. In all other cases, if $\mu_{\D}(\overline{(a,b)})$, respectively $\mu^-_{\D}(\overline{(a,b)})$, is a diameter, it is set to be of the same colour as $\overline{(x,x+n)}$.
}
\end{itemize}
Note that in the case where $\D$ contains two diameters of different colour, $\nc \D \setminus \D$ does not contain any diameters. 
}
\end{itemize}

\end{definition}

\begin{remark}
For any non-crossing diagram $\D$ of Dynkin type $D_n$, the map $\mu_{\D}: \nc \D \to \nc \D$ is a bijection with inverse $\mu_{\D}^-$.
\end{remark}

\begin{example}
Figures \ref{fig:M1} to \ref{fig:M4} provide some examples of mutation of arcs in $\nc \D$. In the pictures, the non-crossing diagram $\D$ is distinguished by thick lines.
\end{example}

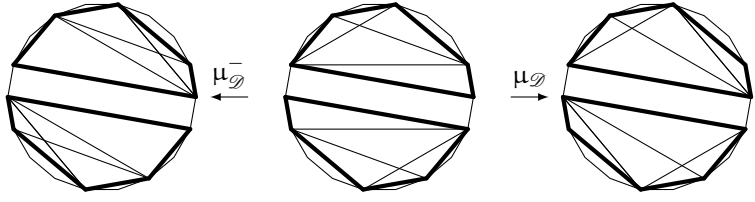
\begin{figure}
\centering{
\begin{tikzpicture}[scale=2.5,cap=round,>=latex, font = \footnotesize, font=\sansmath\sffamily
]
        
	\draw (20:0.5) -- (120:0.5);
	\draw (200:0.5) -- (300:0.5);

	\draw (80:0.5) -- (160:0.5);
	\draw (260:0.5) -- (340:0.5);

	\draw (20:0.5) -- (160:0.5);
	\draw (200:0.5) -- (340:0.5);
	
	\draw[ultra thick] (0:0.5) -- (20:0.5)--(80:0.5);
	\draw[ultra thick]  (80:0.5) -- (120:0.5);
	\draw[ultra thick]  (120:0.5) -- (160:0.5);
	\draw[ultra thick]  (160:0.5) -- (0:0.5);

	\draw[ultra thick]  (180:0.5) -- (200:0.5)--(260:0.5);
	\draw[ultra thick]  (260:0.5) -- (300:0.5);
	\draw[ultra thick]  (300:0.5) -- (340:0.5);
	\draw[ultra thick]  (340:0.5) -- (180:0.5);

        	\draw (0:0.5) -- (20:0.5) -- (40:0.5) -- (60:0.5) -- (80:0.5) -- (100:0.5) -- (120:0.5)  -- (140:0.5) -- (160:0.5) -- (180:0.5) -- (200:0.5) -- (220:0.5) -- (240:0.5) -- (260:0.5) -- (280:0.5) -- (300:0.5) -- (320:0.5) -- (340:0.5) -- (360:0.5) -- (380:0.5);

	\draw[->] (0.7,0) -- (0.9,0) node[midway, above] {$\mu_{\D}$};
	\draw[->] (-0.7,0) -- (-0.9,0) node[midway, above] {$\mu^-_{\D}$};
	
   	\begin{scope}[xshift=-42]

	\draw (0:0.5) -- (80:0.5);
	\draw (180:0.5) -- (260:0.5);

	\draw (0:0.5) -- (80:0.5);
	\draw (180:0.5) -- (260:0.5);	

	\draw (20:0.5) -- (120:0.5);
	\draw (200:0.5) -- (300:0.5);

	\draw (0:0.5) -- (120:0.5);
	\draw (180:0.5) -- (300:0.5);
 
	\draw[ultra thick]  (0:0.5) -- (20:0.5)--(80:0.5);
	\draw[ultra thick]  (80:0.5) -- (120:0.5);
	\draw[ultra thick]  (120:0.5) -- (160:0.5);
	\draw[ultra thick]  (160:0.5) -- (0:0.5);
	
	
	\draw[ultra thick]  (180:0.5) -- (200:0.5)--(260:0.5);
	\draw[ultra thick]  (260:0.5) -- (300:0.5);
	\draw[ultra thick]  (300:0.5) -- (340:0.5);
	\draw[ultra thick]  (340:0.5) -- (180:0.5);

        	\draw (0:0.5) -- (20:0.5) -- (40:0.5) -- (60:0.5) -- (80:0.5) -- (100:0.5) -- (120:0.5)  -- (140:0.5) -- (160:0.5) -- (180:0.5) -- (200:0.5) -- (220:0.5) -- (240:0.5) -- (260:0.5) -- (280:0.5) -- (300:0.5) -- (320:0.5) -- (340:0.5) -- (360:0.5) -- (380:0.5);

     	\end{scope}
	
	\begin{scope}[xshift=42]

	\draw (0:0.5) -- (80:0.5);
	\draw (180:0.5) -- (260:0.5);

	\draw (0:0.5) -- (80:0.5);
	\draw (180:0.5) -- (260:0.5);	

	\draw (80:0.5) -- (160:0.5);
	\draw (260:0.5) -- (340:0.5);

	\draw (0:0.5) -- (120:0.5);
	\draw (180:0.5) -- (300:0.5);
 
	\draw[ultra thick]  (0:0.5) -- (20:0.5)--(80:0.5);
	\draw[ultra thick]  (80:0.5) -- (120:0.5);
	\draw[ultra thick]  (120:0.5) -- (160:0.5);
	\draw[ultra thick]  (160:0.5) -- (0:0.5);
	
	
	\draw[ultra thick]  (180:0.5) -- (200:0.5)--(260:0.5);
	\draw[ultra thick]  (260:0.5) -- (300:0.5);
	\draw[ultra thick]  (300:0.5) -- (340:0.5);
	\draw[ultra thick]  (340:0.5) -- (180:0.5);

        	\draw (0:0.5) -- (20:0.5) -- (40:0.5) -- (60:0.5) -- (80:0.5) -- (100:0.5) -- (120:0.5)  -- (140:0.5) -- (160:0.5) -- (180:0.5) -- (200:0.5) -- (220:0.5) -- (240:0.5) -- (260:0.5) -- (280:0.5) -- (300:0.5) -- (320:0.5) -- (340:0.5) -- (360:0.5) -- (380:0.5);

     	\end{scope}
    \end{tikzpicture}
}
\caption{Mutation with respect to $\D$ of some elements in $\nc \D$ contained in a non-central pair of $\D$-cells}
\label{fig:M1}
\end{figure}

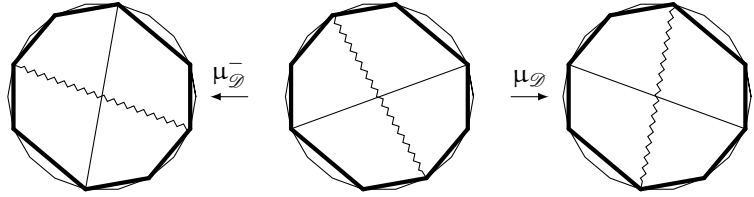
\begin{figure}
\centering{
\begin{tikzpicture}[scale=2.5,cap=round,>=latex, font = \footnotesize, font=\sansmath\sffamily
]
        
	\draw(20:0.5) -- (200:0.5);
	\draw[decorate, decoration={
    zigzag,
    segment length=4,
    amplitude=.9,post=lineto,
    post length=2pt
}] (120:0.5) -- (300:0.5);
	\draw[ultra thick] (20:0.5)--(80:0.5);
	\draw[ultra thick]  (80:0.5) -- (120:0.5);
	\draw[ultra thick]  (120:0.5) -- (160:0.5);

	\draw[ultra thick] (20:0.5)--(340:0.5);
	\draw[ultra thick] (160:0.5)--(200:0.5);

	\draw[ultra thick]  (200:0.5)--(260:0.5);
	\draw[ultra thick]  (260:0.5) -- (300:0.5);
	\draw[ultra thick]  (300:0.5) -- (340:0.5);

        	\draw (0:0.5) -- (20:0.5) -- (40:0.5) -- (60:0.5) -- (80:0.5) -- (100:0.5) -- (120:0.5)  -- (140:0.5) -- (160:0.5) -- (180:0.5) -- (200:0.5) -- (220:0.5) -- (240:0.5) -- (260:0.5) -- (280:0.5) -- (300:0.5) -- (320:0.5) -- (340:0.5) -- (360:0.5) -- (380:0.5);

	\draw[->] (0.7,0) -- (0.9,0) node[midway, above] {$\mu_{\D}$};
	\draw[->] (-0.7,0) -- (-0.9,0) node[midway, above] {$\mu^-_{\D}$};
	
   	\begin{scope}[xshift=-42]
 
	\draw[ultra thick] (20:0.5)--(80:0.5);
	\draw[ultra thick]  (80:0.5) -- (120:0.5);
	\draw[ultra thick]  (120:0.5) -- (160:0.5);
	\draw[ultra thick] (20:0.5)--(340:0.5);
	\draw[ultra thick] (160:0.5)--(200:0.5);

	\draw[decorate, decoration={
    zigzag,
    segment length=4,
    amplitude=.9,post=lineto,
    post length=2pt
}](340:0.5) -- (160:0.5);
	\draw (80:0.5) -- (260:0.5);
	
	
	\draw[ultra thick] (200:0.5)--(260:0.5);
	\draw[ultra thick]  (260:0.5) -- (300:0.5);
	\draw[ultra thick]  (300:0.5) -- (340:0.5);

        	\draw (0:0.5) -- (20:0.5) -- (40:0.5) -- (60:0.5) -- (80:0.5) -- (100:0.5) -- (120:0.5)  -- (140:0.5) -- (160:0.5) -- (180:0.5) -- (200:0.5) -- (220:0.5) -- (240:0.5) -- (260:0.5) -- (280:0.5) -- (300:0.5) -- (320:0.5) -- (340:0.5) -- (360:0.5) -- (380:0.5);

     	\end{scope}
	
	\begin{scope}[xshift=42]

	\draw[ultra thick]  (20:0.5)--(80:0.5);
	\draw[ultra thick]  (80:0.5) -- (120:0.5);
	\draw[ultra thick]  (120:0.5) -- (160:0.5);
	\draw[ultra thick] (20:0.5)--(340:0.5);
	\draw[ultra thick] (160:0.5)--(200:0.5);

	\draw[decorate, decoration={
    zigzag,
    segment length=4,
    amplitude=.9,post=lineto,
    post length=2pt
}](80:0.5) -- (260:0.5);
	\draw (160:0.5) -- (340:0.5);
	
	
	\draw[ultra thick]  (200:0.5)--(260:0.5);
	\draw[ultra thick]  (260:0.5) -- (300:0.5);
	\draw[ultra thick]  (300:0.5) -- (340:0.5);

        	\draw (0:0.5) -- (20:0.5) -- (40:0.5) -- (60:0.5) -- (80:0.5) -- (100:0.5) -- (120:0.5)  -- (140:0.5) -- (160:0.5) -- (180:0.5) -- (200:0.5) -- (220:0.5) -- (240:0.5) -- (260:0.5) -- (280:0.5) -- (300:0.5) -- (320:0.5) -- (340:0.5) -- (360:0.5) -- (380:0.5);

     	\end{scope}
    \end{tikzpicture}
}
\caption{Mutation of some elements in $\nc \D$ with respect to $\D$. If $\D$ contains no diameters, only diameters get mutated to diameters and mutation changes their colour}
\label{fig:M2}
\end{figure}

\begin{figure}
\centering{
\begin{tikzpicture}[scale=2.5,cap=round,>=latex, font = \footnotesize, font=\sansmath\sffamily
]

	
	\fill (0,0) circle (0.5pt);

	\draw[ultra thick] (60:0.5) -- (140:0.5);
	\draw[ultra thick] (200:0.5) -- (240:0.5);
	\draw[ultra thick] (160+180:0.5) -- (20:0.5);
	\draw[ultra thick] (200:0.5) -- (160:0.5);
	\draw[ultra thick] (20:0.5) -- (200:0.5);
	\draw[ultra thick] (160+180:0.5) -- (160:0.5);
	\draw (140:0.5) -- (20:0.5); 
	\draw (140+180:0.5) -- (20+180:0.5); 
	\draw[ultra thick] (140:0.5) -- (160:0.5);
	\draw[ultra thick] (140+180:0.5) -- (160+180:0.5);
	
	\draw[->] (0.7,0) -- (0.9,0) node[midway, above] {$\mu_{\D}$};
	\draw[->] (-0.7,0) -- (-0.9,0) node[midway, above] {$\mu^-_{\D}$};


	\draw[ultra thick] (20:0.5) -- (60:0.5);
	\draw[ultra thick] (60+180:0.5) -- (140+180:0.5);

        	\draw (0:0.5) -- (20:0.5) -- (40:0.5) -- (60:0.5) -- (80:0.5) -- (100:0.5) -- (120:0.5)  -- (140:0.5) -- (160:0.5) -- (180:0.5) -- (200:0.5) -- (220:0.5) -- (240:0.5) -- (260:0.5) -- (280:0.5) -- (300:0.5) -- (320:0.5) -- (340:0.5) -- (360:0.5) -- (380:0.5);

	\begin{scope}[xshift = 42]
	
	\fill (0,0) circle (0.5pt);

	\draw[ultra thick] (60:0.5) -- (140:0.5);
	\draw[ultra thick] (200:0.5) -- (240:0.5);
	\draw[ultra thick] (160+180:0.5) -- (20:0.5);
	\draw[ultra thick] (200:0.5) -- (160:0.5);
	\draw[ultra thick] (20:0.5) -- (200:0.5);
	\draw[ultra thick] (160+180:0.5) -- (160:0.5);
	\draw (60:0.5) -- (160:0.5); 
	\draw (240:0.5) -- (340:0.5); 
	\draw[ultra thick] (140:0.5) -- (160:0.5);
	\draw[ultra thick] (140+180:0.5) -- (160+180:0.5);


	\draw[ultra thick] (20:0.5) -- (60:0.5);
	\draw[ultra thick] (60+180:0.5) -- (140+180:0.5);

        	\draw (0:0.5) -- (20:0.5) -- (40:0.5) -- (60:0.5) -- (80:0.5) -- (100:0.5) -- (120:0.5)  -- (140:0.5) -- (160:0.5) -- (180:0.5) -- (200:0.5) -- (220:0.5) -- (240:0.5) -- (260:0.5) -- (280:0.5) -- (300:0.5) -- (320:0.5) -- (340:0.5) -- (360:0.5) -- (380:0.5);

	\end{scope}
\begin{scope}[xshift = -42]
	
	\fill (0,0) circle (0.5pt);

	\draw[ultra thick] (60:0.5) -- (140:0.5);
	\draw[ultra thick] (200:0.5) -- (240:0.5);
	\draw[ultra thick] (160+180:0.5) -- (20:0.5);
	\draw[ultra thick] (200:0.5) -- (160:0.5);
	\draw[ultra thick] (20:0.5) -- (200:0.5);
	\draw[ultra thick] (160+180:0.5) -- (160:0.5);
	\draw (60:0.5) -- (240:0.5); 
	\draw[ultra thick] (140:0.5) -- (160:0.5);
	\draw[ultra thick] (140+180:0.5) -- (160+180:0.5);

	
	\draw[ultra thick] (20:0.5) -- (60:0.5);
	\draw[ultra thick] (60+180:0.5) -- (140+180:0.5);

        	\draw (0:0.5) -- (20:0.5) -- (40:0.5) -- (60:0.5) -- (80:0.5) -- (100:0.5) -- (120:0.5)  -- (140:0.5) -- (160:0.5) -- (180:0.5) -- (200:0.5) -- (220:0.5) -- (240:0.5) -- (260:0.5) -- (280:0.5) -- (300:0.5) -- (320:0.5) -- (340:0.5) -- (360:0.5) -- (380:0.5);

	\end{scope}
       \end{tikzpicture}
\\

\begin{tikzpicture}[scale=2.5,cap=round,>=latex, font=\sansmath\sffamily
]

	
	\fill (0,0) circle (0.5pt);

	\draw[ultra thick] (60:0.5) -- (140:0.5);
	\draw[ultra thick] (200:0.5) -- (240:0.5);
	\draw[ultra thick] (160+180:0.5) -- (20:0.5);
	\draw[ultra thick] (200:0.5) -- (160:0.5);
	\draw[ultra thick] (20:0.5) -- (200:0.5);
	\draw[ultra thick] (160+180:0.5) -- (160:0.5);
	\draw (140:0.5) -- (320:0.5); 
	\draw[ultra thick] (140:0.5) -- (160:0.5);
	\draw[ultra thick] (140+180:0.5) -- (160+180:0.5);
	
	\draw[->] (0.7,0) -- (0.9,0) node[midway, above] {$\mu_{\D}$};
	\draw[->] (-0.7,0) -- (-0.9,0) node[midway, above] {$\mu^-_{\D}$};


	\draw[ultra thick] (20:0.5) -- (60:0.5);
	\draw[ultra thick] (60+180:0.5) -- (140+180:0.5);

        	\draw (0:0.5) -- (20:0.5) -- (40:0.5) -- (60:0.5) -- (80:0.5) -- (100:0.5) -- (120:0.5)  -- (140:0.5) -- (160:0.5) -- (180:0.5) -- (200:0.5) -- (220:0.5) -- (240:0.5) -- (260:0.5) -- (280:0.5) -- (300:0.5) -- (320:0.5) -- (340:0.5) -- (360:0.5) -- (380:0.5);

	\begin{scope}[xshift = 42]
	
	\fill (0,0) circle (0.5pt);

	\draw[ultra thick] (60:0.5) -- (140:0.5);
	\draw[ultra thick] (200:0.5) -- (240:0.5);
	\draw[ultra thick] (160+180:0.5) -- (20:0.5);
	\draw[ultra thick] (200:0.5) -- (160:0.5);
	\draw[ultra thick] (20:0.5) -- (200:0.5);
	\draw[ultra thick] (160+180:0.5) -- (160:0.5);
	\draw (20:0.5) -- (160:0.5); 
	\draw (200:0.5) -- (340:0.5); 
	\draw[ultra thick] (140:0.5) -- (160:0.5);
	\draw[ultra thick] (140+180:0.5) -- (160+180:0.5);


	\draw[ultra thick] (20:0.5) -- (60:0.5);
	\draw[ultra thick] (60+180:0.5) -- (140+180:0.5);

        	\draw (0:0.5) -- (20:0.5) -- (40:0.5) -- (60:0.5) -- (80:0.5) -- (100:0.5) -- (120:0.5)  -- (140:0.5) -- (160:0.5) -- (180:0.5) -- (200:0.5) -- (220:0.5) -- (240:0.5) -- (260:0.5) -- (280:0.5) -- (300:0.5) -- (320:0.5) -- (340:0.5) -- (360:0.5) -- (380:0.5);

	\end{scope}
\begin{scope}[xshift = -42]
	
	\fill (0,0) circle (0.5pt);

	\draw[ultra thick] (60:0.5) -- (140:0.5);
	\draw[ultra thick] (200:0.5) -- (240:0.5);
	\draw[ultra thick] (160+180:0.5) -- (20:0.5);
	\draw[ultra thick] (200:0.5) -- (160:0.5);
	\draw[ultra thick] (20:0.5) -- (200:0.5);
	\draw[ultra thick] (160+180:0.5) -- (160:0.5);
	\draw (60:0.5) -- (160:0.5); 
	\draw (240:0.5) -- (340:0.5); 
	\draw[ultra thick] (140:0.5) -- (160:0.5);
	\draw[ultra thick] (140+180:0.5) -- (160+180:0.5);

	
	\draw[ultra thick] (20:0.5) -- (60:0.5);
	\draw[ultra thick] (60+180:0.5) -- (140+180:0.5);

        	\draw (0:0.5) -- (20:0.5) -- (40:0.5) -- (60:0.5) -- (80:0.5) -- (100:0.5) -- (120:0.5)  -- (140:0.5) -- (160:0.5) -- (180:0.5) -- (200:0.5) -- (220:0.5) -- (240:0.5) -- (260:0.5) -- (280:0.5) -- (300:0.5) -- (320:0.5) -- (340:0.5) -- (360:0.5) -- (380:0.5);

	\end{scope}
       \end{tikzpicture}
}
\caption{Mutation of some elements in $\nc \D$ with respect to $\D$. If $\D$ contains diameters, pairs of arcs might get glued together to diameters and diameters might get split up into pairs of arcs when mutating}
\label{fig:M3}
\end{figure}
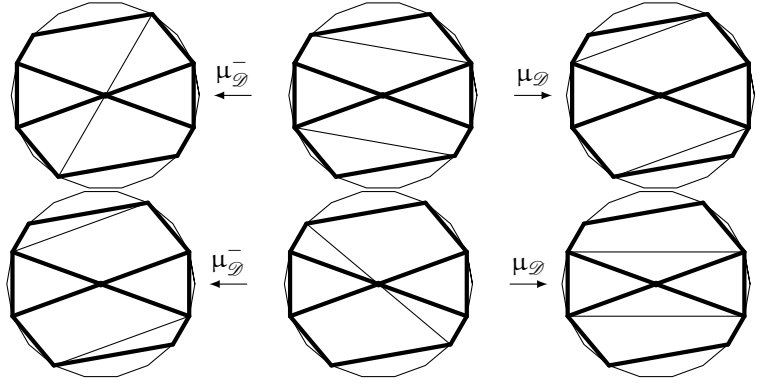

\begin{figure}
\centering{
\begin{tikzpicture}[scale=2.5,cap=round,>=latex, font = \footnotesize, font=\sansmath\sffamily
]
        
\begin{scope}[xshift = 42]
	
	\fill (0,0) circle (0.5pt);

	\draw[ultra thick] (60:0.5) -- (140:0.5);
	\draw[ultra thick] (200:0.5) -- (240:0.5);
	\draw[ultra thick] (160+180:0.5) -- (20:0.5);
	\draw[ultra thick] (200:0.5) -- (160:0.5);
	\draw (20:0.5) -- (200:0.5);
	\draw[ultra thick] (160+180:0.5) -- (160:0.5);
	\draw (60:0.5) -- (60+180:0.5); 
	\draw[ultra thick] (140:0.5) -- (160:0.5);
	\draw[ultra thick] (140+180:0.5) -- (160+180:0.5);
	
	\draw[ultra thick] (20:0.5) -- (60:0.5);
	\draw[ultra thick] (60+180:0.5) -- (140+180:0.5);

        	\draw (0:0.5) -- (20:0.5) -- (40:0.5) -- (60:0.5) -- (80:0.5) -- (100:0.5) -- (120:0.5)  -- (140:0.5) -- (160:0.5) -- (180:0.5) -- (200:0.5) -- (220:0.5) -- (240:0.5) -- (260:0.5) -- (280:0.5) -- (300:0.5) -- (320:0.5) -- (340:0.5) -- (360:0.5) -- (380:0.5);

\end{scope}

\begin{scope}[xshift = -42]
	
	\fill (0,0) circle (0.5pt);

	\draw[ultra thick] (60:0.5) -- (140:0.5);
	\draw[ultra thick] (200:0.5) -- (240:0.5);
	\draw[ultra thick] (160+180:0.5) -- (20:0.5);
	\draw[ultra thick] (200:0.5) -- (160:0.5);
	\draw (140:0.5) -- (140+180:0.5);
	\draw[ultra thick] (160+180:0.5) -- (160:0.5);
	\draw (140:0.5) -- (160+180:0.5); 
	\draw (140+180:0.5) -- (160:0.5); 
	\draw[ultra thick] (140:0.5) -- (160:0.5);
	\draw[ultra thick] (140+180:0.5) -- (160+180:0.5);

	\draw[ultra thick] (20:0.5) -- (60:0.5);
	\draw[ultra thick] (60+180:0.5) -- (140+180:0.5);

        	\draw (0:0.5) -- (20:0.5) -- (40:0.5) -- (60:0.5) -- (80:0.5) -- (100:0.5) -- (120:0.5)  -- (140:0.5) -- (160:0.5) -- (180:0.5) -- (200:0.5) -- (220:0.5) -- (240:0.5) -- (260:0.5) -- (280:0.5) -- (300:0.5) -- (320:0.5) -- (340:0.5) -- (360:0.5) -- (380:0.5);

\end{scope}

	\fill (0,0) circle (0.5pt);

	\draw[ultra thick] (60:0.5) -- (140:0.5);
	\draw[ultra thick] (200:0.5) -- (240:0.5);
	\draw[ultra thick] (160+180:0.5) -- (20:0.5);
	\draw[ultra thick] (200:0.5) -- (160:0.5);
	\draw[decorate, decoration={
    zigzag,
    segment length=4,
    amplitude=.9,post=lineto,
    post length=2pt
}] (155:0.5) -- (345:0.5);
	\draw[ultra thick] (160+180:0.5) -- (160:0.5);
	\draw (20:0.5) -- (160:0.5);  
	\draw (200:0.5) -- (340:0.5);  
	\draw[ultra thick] (140:0.5) -- (160:0.5);
	\draw[ultra thick] (140+180:0.5) -- (160+180:0.5);
	
	\draw[->] (0.7,0) -- (0.9,0) node[midway, above] {$\mu_{\D}$};
	\draw[->] (-0.7,0) -- (-0.9,0) node[midway, above] {$\mu^-_{\D}$};

	\draw[ultra thick] (20:0.5) -- (60:0.5);
	\draw[ultra thick] (60+180:0.5) -- (140+180:0.5);

        	\draw (0:0.5) -- (20:0.5) -- (40:0.5) -- (60:0.5) -- (80:0.5) -- (100:0.5) -- (120:0.5)  -- (140:0.5) -- (160:0.5) -- (180:0.5) -- (200:0.5) -- (220:0.5) -- (240:0.5) -- (260:0.5) -- (280:0.5) -- (300:0.5) -- (320:0.5) -- (340:0.5) -- (360:0.5) -- (380:0.5);

	\end{tikzpicture}
}
\caption{Mutation of some elements in $\nc \D$ with respect to $\D$. In this example, $\D$ contains exactly one diameter}
\label{fig:M4}
\end{figure}
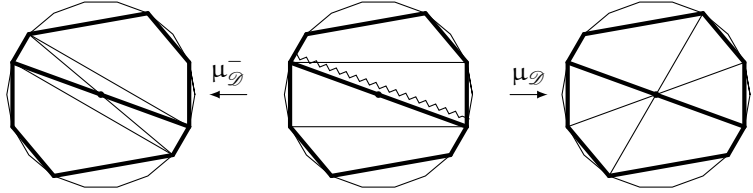

The following definition mirrors the definition of $D$-mutation pairs in triangulated categories by Iyama and Yoshino in \cite{IY}, cf Definition \ref{D:mutation}.

\begin{definition}[$\D$-mutation pair]
We call a pair of diagrams $(\X,\X')$ of Dynkin type $D_n$ with $\X, \X' \subset \nc \D$ a {\em $\D$-mutation pair} if $\D \subset \X' \subset \mu^-_{\D}(\X)$ and $\D \subset \X \subset \mu_{\D}(\X')$.
\end{definition} 

\begin{remark}
Since $\mu_{\D}$ is a bijection on $\nc \D$ with inverse $\mu_{\D}^-$, for any $\D$-mutation pair $(\X,\X')$ we have $\X = \mu_{\D}(\X')$ and $\X' = \mu^-_{\D}(\X)$.
\end{remark}

\begin{lemma}\label{mutation changes colour}
Mutation changes the colour of diameters. I.e.\;let $\overline{(a,b)}$ be a diameter in $\nc \D \setminus \D$ such that $\mu^-_{\D}(\overline{(a,b)})$ (respectively $\mu_{\D}(\overline{(a,b)})$) is also a diameter. Then $\mu^-_{\D}(\overline{(a,b)})$ (respectively $\mu_{\D}(\overline{(a,b)})$ is of different colour than $\overline{(a,b)}$.
\end{lemma}

\begin{proof}
If $\D$ contains no diameters this follows directly from Definition \ref{combinatorial mutation}. If $\D$ contains two diameters of different colour, then $\nc \D \setminus \D$ contains no diameters, so the statement is trivial. Thus we only have to consider the case where $\D$ contains at least one diameter and all its diameters are of the same colour.

We first note that if $r_{\D}(\overline{(a,b)})$ is a pair of radii then $r_{\D}(\mu_{\D}(\overline{(a,b)}))$ is not: since the ending vertices of $r_{\D}(\overline{(a,b)})$ and $r_{\D}(\mu_{\D}(\overline{(a,b)}))$ are pairwise distinct, at most one of them can be the central vertex $c$. 

If $\D$ contains more than one diameter of the same colour, then every diameter in $\nc \D \setminus \D$ gets mapped to a pair of radii under the map $r_{\D}$. Thus, in this case, diameters do not get mutated to diameters.

It remains to check the case where $\D$ contains exactly one diameter. Assume that both $\overline{(a,b)}$ in $\nc \D \setminus \D$ and $\mu_{\D}(\overline{(a,b)})$, respectively $\mu^-_{\D}(\overline{(a,b)})$, are diameters. Since $r_{\D}(\overline{(a,b)})$ and $r_{\D}(\mu_{\D}(\overline{(a,b)}))$, respectively $r_{\D}(\mu^-_{\D}(\overline{(a,b)}))$, cannot both be pairs of radii, precisely one of them has to be a diameter. However, there is only one diameter in $r_{\D}(\nc \D \setminus \D)$ and it is of different colour to all the pairs of radii in $r_{\D}(\nc \D \setminus \D)$. Therefore mutation changes colour.
\end{proof}

\begin{remark}\label{diameters to arcs}
If a diameter $\overline{(a,b)}$ gets mutated to a pair of arcs $\mu_{\D}(\overline{(a,b)})$ or $\mu^-_{\D}(\overline{(a,b)})$, then $r_{\D}(\overline{(a,b)}) = \overline{(d_i,d_j)}$ is a pair of radii. The way we may think about this is that only pairs of radii may get split up into pairs of arcs and diameters have to stay "whole" (at least for one mutation step). Indeed, if $\mu^{(-)}_{\D}(\overline{(a,b)})=r^{-1}_{\D}\overline{(d_{i\pm 1},d_{j \pm 1})}$ is a pair of arcs, then the four vertices $d_{i\pm 1},d_{j \pm 1}, (d_{i\pm 1}+n)$ and $(d_{j \pm 1}+n)$ are pairwise distinct. By rotation invariance of $\D$, the element $\overline{(d_i,d_j)}$ cannot be a diameter.
\end{remark}

\section{A combinatorial model for mutation of torsion pairs in the cluster category of Dynkin type $D_n$}
\label{A combinatorial model for mutation}

Our goal is to give a combinatorial interpretation for mutation of torsion pairs in the cluster category $\CC_{D_n}$ by defining mutation of Ptolemy diagrams of Dynkin type $D_n$. Since $\CC_{D_n}$ is $2$-Calabi-Yau and contains only finitely many indecomposable objects (up to isomorphism), any subcategory is functorially finite and satisfies $\tau D = \Sigma D$. According to Zhou and Zhu's result (cf.\;Theorem \ref{Zhou-Zhu}), mutation of a torsion pair $(X,Y)$ in $\CC_{D_n}$ is thus defined with respect to every subcategory $D \subset X \cap (\Sigma^{-1} X)^\perp$. If $X$ corresponds to the diagram $\X$ of Dynkin type $D_n$, the subcategory $X \cap (\Sigma^{-1} X)^\perp$ corresponds to the diagram $\X \cap \nc \X$ of those arcs in $\X$ that do not cross any other arcs in $\X$. In analogy with mutation of torsion pairs on the categorical level, we want to define mutation of the Ptolemy diagram $\X$ of Dynkin type $D_n$ with respect to subdiagrams of $\X \cap \nc \X$. In particular, any subdiagram of $\X \cap \nc \X$ is a non-crossing diagram of Dynkin type $D_n$.

\begin{definition}[mutation of Ptolemy diagrams of Dynkin type $D_n$]\label{diagram mutation}
Let $\X$ be a Ptolemy diagram of Dynkin type $D_n$ with a subdiagram $\D \subset \X \cap \nc \X$. We define the {\em $\D$-mutations of $\X$} to be the diagrams
\begin{eqnarray*}
\mu_{\D}(\X) &=& \{\mu_{\D}(\overline{(a,b)})|\overline{(a,b)} \in \X\} \text{ and }\\
\mu^-_{\D}(\X) &=& \{\mu^-_{\D}(\overline{(a,b)})|\overline{(a,b)} \in \X\}.
\end{eqnarray*}
\end{definition}

Mutation of Ptolemy diagrams of Dynkin type $D_n$ as described in Definition \ref{diagram mutation} provides a combinatorial model for mutation of torsion pairs in the cluster category of Dynkin type $D_n$. For the proof of this fact, which is stated more generally in Theorem \ref{mutations agree general}, we calculate extensions between certain objects in the cluster category $\CC_{D_n}$. We first introduce a few technicalities to make the calculations easier. 

\begin{lemma} \label{dimension of Ext}
Let $\D$ be a non-crossing diagram and let $\overline{(a,b)} \in \nc \D$. Then we have 
$$
\dim(\Ext_{\CC_{D_n}}^1(m_{\overline{(a,b)}},m_{\mu^-_{\D}(\overline{(a,b)})}) = 1.
$$
\end{lemma}

\begin{proof}
By Lemma \ref{intersection number}, the dimension of the extension space of two indecomposable objects is equal to the number of times the corresponding pairs of arcs cross (cf. Definition \ref{crossing arcs}). Let $\overline{\langle d_1, \ldots, d_k\rangle}$ be the pair of $\D$-cells containing $\overline{(a,b)}$ and $\mu^-_{\D}(\overline{(a,b)})$ with $\overline{(a,b)} = r^{-1}_{\D}(\overline{(d_i,d_j)})$ and $\mu^-_{\D}(\overline{(a,b)}) = r^{-1}_{\D}(\overline{(d_{i-1},d_{j-1})})$. The vertices $d_i, d_{j-1}, d_j$ and $d_{i-1}$ appear in this order in an anti clockwise direction on the border of $\langle d_1, \ldots, d_k \rangle$. The two arcs $(d_i,d_j)$ and $(d_{i-1},d_{j-1})$ thus cross. We distinguish the following cases.
\begin{itemize}
\item{If both $\overline{(a,b)}$ and $\mu^-_{\D}(\overline{(a,b)})$ are diameters then by Lemma \ref{mutation changes colour} they are of different colour, so they cross once.}
\item{If one of $\overline{(a,b)}$ and $\mu^-_{\D}(\overline{(a,b)})$ is a diameter and the other one is a pair of arcs it follows directly from Definition \ref{crossing arcs} that they cross once.}
\item{Now consider the case where both $\overline{(a,b)}$ and $\mu^-_{\D}(\overline{(a,b)})$ are pairs of arcs. We show that $\overline{(d_i,d_j})$ and $\overline{(d_{i-1},d_{j-1})}$ cannot cross twice, i.e.\;we show that if the arc $(d_i,d_j)$ crosses $(d_{i-1},d_{j-1})$ it cannot cross its partner $(d_{i-1}+n,d_{j-1}+n)$. Assume, for a contradiction, that it does and without loss of generality assume $d_j < d_i+n$. 
\begin{figure}
\centering{
\begin{tikzpicture}[scale=2.5,cap=round,>=latex, font = \footnotesize, font=\sansmath\sffamily
]
        
	\node (d1) at (60:0.57cm){$d_i$};
	\node (d2) at (150:0.57cm){$d_j$};
	\node (dk) at (310:0.57cm){$d_{i-1}$};
	\node (dk+1) at (90:0.57cm){$d_{j-1}$};
	
	\draw (60:0.5) -- (150:0.5);
	\draw (240:0.5) -- (330:0.5);
	\draw (310:0.5) -- (90:0.5);
	\draw (130:0.5cm) -- (270:0.5cm);
	\draw[dashed] (310:0.5cm) -- (60:0.5cm);
	\draw[dashed] (240:0.5cm) -- (130:0.5cm);
	
        	\draw (0,0) circle(0.5cm);
\begin{scope}[xshift = 42]
	\node (d1) at (60:0.57cm){$d_i$};
	\node (d2) at (150:0.57cm){$d_j$};
	\node (dk) at (130:0.57cm){$d_{i-1}$};
	\node (dk+1) at (270:0.57cm){$d_{j-1}$};
	
	\draw (60:0.5) -- (150:0.5);
	\draw (240:0.5) -- (330:0.5);
	\draw (310:0.5) -- (90:0.5);
	\draw (130:0.5cm) -- (270:0.5cm);
	\draw[dashed] (270:0.5cm) -- (150:0.5cm);
	\draw[dashed] (450:0.5cm) -- (330:0.5cm);

        	\draw (0,0) circle(0.5cm);
  \end{scope}
  \end{tikzpicture}
}
\caption{The pairs of arcs $\overline{(d_i,d_j)}$ and $\overline{(d_{i-1},d_{j-1})}$ cannot cross twice}
\label{fig:NC}
\end{figure}
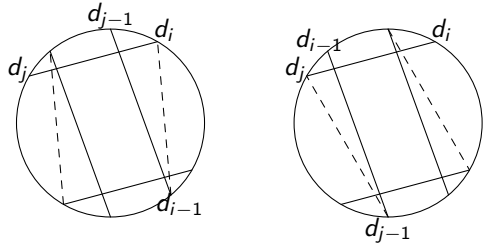
Then either both $d_{j-1}$ and $d_{i-1}+n$ or both $d_{i-1}$ and $d_{j-1}+n$  lie between $d_i$ and $d_j$ in a clockwise direction, cf.\;Figure \ref{fig:NC}. In the first case this would imply that $\overline{(d_i,d_{i-1})}$ and $\overline{(d_i,d_j)}$ cross and in the second case that $\overline{(d_j,d_{j-1})}$ and $\overline{(d_i,d_j)}$ cross. However, $\overline{(d_i,d_{i-1})}$ and $\overline{(d_j,d_{j-1})}$ both lie in $r_{\D}(\D)$ and $\overline{(d_i,d_j)}$ is a pair of arcs in $\X$, so this contradicts the fact that $\D \subset \nc \X$.
}
\end{itemize}
\end{proof}

We now want to find distinguished triangles of the form
$$
m_{\overline{(a,b)}} \to d \to m_{\mu^-_{\D}(\overline{(a,b)})} \to \Sigma m_{\overline{(a,b)}}
$$
in $\CC_{D_n}$. To do this, we use methods introduced by Buan, Marsh, Reineke, Reiten and Todorov in \cite{BMRRT}.
Recall that the cluster category $\CC_{D_n}$ is an orbit category of the bounded derived category $D^b(kD_n)$. By Proposition 1.6 in \cite{BMRRT} the objects in the cluster category $\CC_{D_n}$ are either induced by $kD_n$-modules or by shifts of projective modules. We restrict the coordinate system from Iyama's paper \cite{Iyama} on the Auslander-Reiten quiver of the derived category to the Auslander-Reiten quiver $\AR(\Mod kD_n)$ of the module category, cf. Figure \ref{fig:coordinate-system}.

\begin{centering}
\begin{figure}
\begin{tikzpicture} [scale=1.4, font = \footnotesize, font=\sansmath\sffamily
]

\node (a1) at (0,0) {$[0,2]$};
\node (a2) at (1,1) {$[0,3]$};
\node (af) at (3,3) {$\ldots$};
\node (a3) at (2,2) {$[0,4]$};
\node (a+) at (4,4) {$[0,n]_+$};
\node (a-) at (4,3) {$[0,n]_-$};

\draw[->] (a1) -- (a2);
\draw[->] (a2) -- (a3);
\draw[dotted] (a3) -- (af);
\draw[->] (af) -- (a+);
\draw[->] (af) -- (a-);

\node (b1) at (2,0) {$[1,3]$};
\node (b2) at (3,1) {$[1,4]$};
\node (bf) at (5,3) {$\ldots$};
\node (b3) at (4,2) {$[1,5]$};
\node (b+) at (6,4) {$[1,n+1]_+$};
\node (b-) at (6,3) {$[1,n+1]_-$};

\draw[->] (b1) -- (b2);
\draw[->] (b2) -- (b3);
\draw[dotted] (b3) -- (bf);
\draw[->] (bf) -- (b+);
\draw[->] (bf) -- (b-);

\node (c1) at (4,0) {};
\node (c2) at (5,1) {};
\node (cf) at (7,3) {};
\node (c3) at (6,2) {};
\node (c+) at (8,4) {};
\node (c-) at (8,3) {};

\node (e1) at (6,0)  {$[n-2,n]$};
\node (e2) at (7,1) {$[n-2,n+1]$};
\node (ef) at (9,3) {$\ldots$};
\node (e3) at (8,2) {$[n-2,n+2]$};
\node (e+) at (10,4) {$[n-2,2n-1]_+$};
\node (e-) at (10,3) {$[n-2,2n-1]_-$};

\draw[->] (e1) -- (e2);
\draw[->] (e2) -- (e3);
\draw[dotted] (e3) -- (ef);
\draw[->] (ef) -- (e+);
\draw[->] (ef) -- (e-);

\draw[->] (c1) -- (c2);
\draw[->] (c2) -- (c3);
\draw[dotted] (c3) -- (cf);
\draw[->] (cf) -- (c+);
\draw[->] (cf) -- (c-);

\draw[->] (a2) -- (b1);
\draw[dotted] (b2) -- (c1);
\draw[->] (a3) -- (b2);
\draw[dotted] (b3) -- (c2);
\draw[->] (a+) -- (bf);
\draw[->] (a-) -- (bf);
\draw[dotted] (b+) -- (cf);
\draw[dotted] (b-) -- (cf);

\draw[->] (af) -- (b3);
\draw[dotted] (bf) -- (c3);

\draw[->] (c2) -- (e1);
\draw[->] (c3) -- (e2);
\draw[->] (cf) -- (e3);

\draw[->] (c-) -- (ef);
\draw[->] (c+) -- (ef);

\end{tikzpicture}

\caption{Coordinate system on $\AR(\Mod kD_n)$}
\label{fig:coordinate-system}
\end{figure}
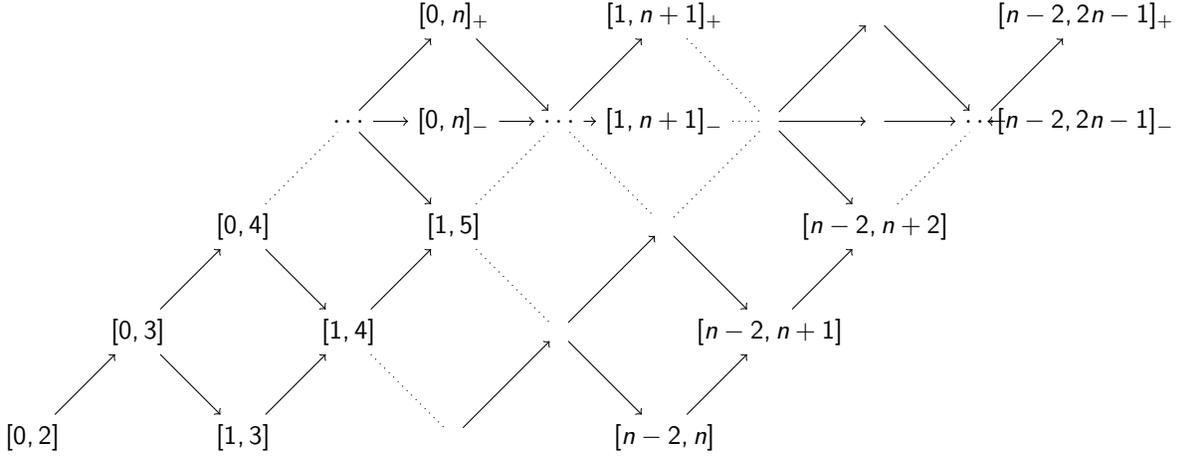
\end{centering}
We will denote a representative of the isomorphism class of indecomposable modules at the vertex $[i,j]$ or $[i,i+n]_\pm$ by $M_{[i,j]}$ or $M_{[i,i+n]_\pm}$. By abuse of notation we sometimes omit the signs and simply write $[i,i+n]$. The projective modules are those in the first slice of $\AR(\Mod kD_n)$, i.e.\;the modules of the form $M_{[0,j]}$ for $2 \leq j \leq n$. The $kD_n$-module $M_{[i,j]}$ induces the object $m_{b([i,j])}$ in $\CC_{D_n}$ associated to the element $b([i,j]) \in \A(\p_{2n})$, where $b$ is the bijection from Definition \ref{D:bijection}.

Let us first recall the notion of starting and ending frame in $\Mod kD_n$. We denote the set of vertices of $\AR(\Mod kD_n)$ by $\AR(kD_n)_0$.

\begin{definition}[\cite{BMRRT}, Definition 8.4]
Let $[i,j] \in \AR(kD_n)_0$ be a vertex of the Auslander-Reiten quiver of the module category $\Mod kD_n$. The {\em starting frame} $F_s([i,j])$ and {\em ending frame} $F_e([i,j])$) of the vertex $[i,j]$ are defined as follows:
$$
F_s([i,j]) = \Bigg\{ [k,l] \in \AR(kD_n)_0\Bigg| \begin{aligned}
\Hom_{kD_n}(M_{[i,j]}, M_{[k,l]}) \neq 0 \text{ and}\\
\Hom_{kD_n}(M_{[i,j]},\tau M_{[k,l]}) = 0
\end{aligned}
\Bigg\}
$$
$$
F_e([i,j]) = \Bigg\{[k,l] \in \AR(kD_n)_0\Bigg|\begin{aligned}
\Hom_{kD_n}(M_{[k,l]},M_{[i,j]}) \neq 0 \text{ and}\\
\Hom_{kD_n}(\tau^{-1}M_{[k,l]},M_{[i,j]}) = 0
\end{aligned}
\Bigg\}
$$
\end{definition}

\begin{lemma}[\cite{BMRRT}, Corollary 8.5]\label{intersect frames}
Let $M_{[i,j]}$ and $M_{[k,l]}$ be indecomposable objects in $\Mod kD_n$ such that $\Ext^1_{kD_n}(M_{[i,j]},M_{[k,l]})$ is one-dimensional. Then the (up to isomorphism) unique non-trivial extension of $M_{[i,j]}$ by $M_{[k,l]}$ is the direct sum of one copy of each indecomposable object corresponding to a vertex in the intersection $F_s([k,l]) \cap F_e([i,j])$:
$$
\Ext^1_{kD_n}(M_{[i,j]},M_{[k,l]}) \cong \bigoplus_{[a,b] \in F_s([k,l]) \cap F_e([i,j])} M_{[a,b]}.
$$
\end{lemma}

The starting and ending frames can be worked out using the tables from Section 1.3 in Bongartz's paper \cite{Bongartz}, cf.\;also Buan, Marsh, Reineke, Reiten and Todorov's paper \cite{BMRRT}. 
For a vertex $[i,j]$ of $\AR(\Mod kD_n)$ with $i < j < i+n$, they are given by: 
$$
F_s([i,j]) = \begin{aligned}
&\{[i,k] \in \AR(kD_n)_0| j \leq k \leq i+n-1\}\\
&\cup  \{[i,i+n]_+ , [i,i+n]_- \}\\ 
&\cup \{[k,j]\in \AR(kD_n)_0| i \leq k \leq j-2\} \\
 &\cup \{[k,i+n]\in \AR(kD_n)_0| j \leq k \leq  i+n-2\}
\end{aligned}
$$
and
$$
F_e((i,j)) = 
\begin{aligned}
&\{[i,k]\in \AR(kD_n)_0|i+2 \leq k \leq j\} \\ & \cup  \{[k,j]\in \AR(kD_n)_0| j-n+1 \leq k \leq i\}\\ & \cup \{ [j-n,j]_+ ,[j-n,j]_-\}\\
&  \cup  \{[j-n,k]\in \AR(kD_n)_0|j-n+2 \leq k \leq i\},
\end{aligned}
$$
cf.\;Figures \ref{fig:AR1} and \ref{fig:AR2}.
For the vertices $[i,i+n]_\pm$ the starting and ending frames are given by
$$
F_s([i,i+n]_\pm) = 
\begin{aligned}
& \{[k,i+n] \in \AR(kD_n)_0|i < k \leq i+n-2\} \\ & \cup \{[k,k+n]_\pm|i\leq k \text{ and } k - i = 0 \mod 2\} \\ & \cup \{[k,k+n]_\mp|i\leq k \text{ and } k - i = 1 \mod 2\}
\end{aligned}
$$
and
$$
F_e([i,i+n]_\pm) = 
\begin{aligned}
& \{[i,k] \in \AR(kD_n)_0|i+2 \leq k < i+n\} \\ & \cup \{[k,k+n]_\pm|i\geq k \text{ and } i - k = 0 \mod 2\} \\ & \cup \{[k,k+n]_\mp|i \geq k \text{ and } i - k = 1 \mod 2\},
\end{aligned}
$$
cf.\;Figures \ref{fig:AR3} and \ref{fig:AR4}.
\begin{figure}
\begin{centering}
\begin{tikzpicture} [scale=1, font = \footnotesize, font=\sansmath\sffamily
]

\fill[black!20] (3,2) -- (5,4) -- (9,4) -- (10.5,2.5) -- (8.5,0.5) -- (7,2) -- (5,0) -- (3,2) -- cycle;

\draw[dotted] (0,0) -- (4,4);
\draw[dotted] (4,4) -- (12,4);
\draw[dotted] (12,4) -- (8,0);
\draw[dotted] (0,0) -- (8,0);

\node (a) at (3,2) {$[i,j]$};
\node[label=right:{$[i,i+n]_+$}] (b) at (5,4) {};
\node (c) at (5,0) {$[j-2,j]$};
\node[label=right:{$[i,i+n]_-$}] (d) at (5,3.5) {};
\node (f) at (7.5,1.5) {$[j,i+n]$};

\draw (3,2)-- (5,4);
\draw (4.5,3.5) -- (5,3.5);
\draw (3,2) -- (5,0);
\draw (8.5,0.5) -- (7.5,1.5);

\end{tikzpicture}
\caption{The lines mark the starting frame of the vertex $[i,j]$ in $\AR(\Mod kD_n)$. The area into which there are non-trivial morphisms from the module $M_{[i,j]}$ is marked in grey}
\label{fig:AR1}
\end{centering}
\end{figure}
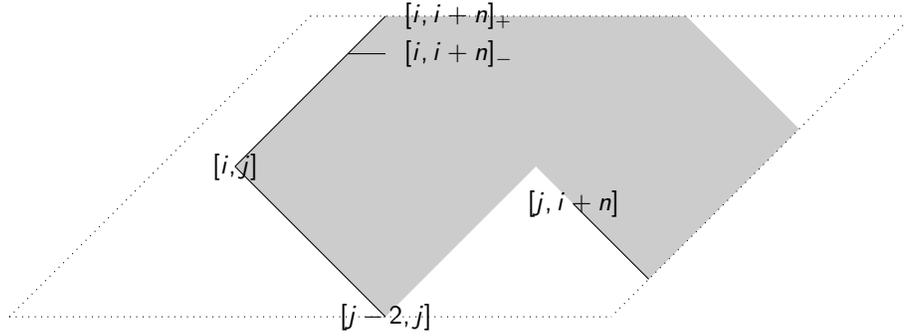

\begin{figure}
\begin{centering}
\begin{tikzpicture} [scale=1, font = \footnotesize, font=\sansmath\sffamily
]
\fill[black!20] (2,2) -- (4,4) -- (7,4) -- (9,2) -- (7,0) -- (5,2) -- (3,0) -- (1.5,1.5) -- cycle;
\draw[dotted] (0,0) -- (4,4);
\draw[dotted] (4,4) -- (12,4);
\draw[dotted] (12,4) -- (8,0);
\draw[dotted] (0,0) -- (8,0);

\node (a) at (9,2) {$[i,j]$};
\node[label=left:{$[j-n,j]_+$}] (b) at (7,4) {};
\node (c) at (7,0) {$[i,i+2]$};
\node[label=left:{$[j-n,j]_-$}] (d) at (7,3.5) {};
\node (e) at (3,0) {$[j-n,j-n+2]$};
\node (f) at (4.5,1.5) {$[j-n,i]$};

\draw (9,2)-- (7,4);
\draw (7.5,3.5) -- (7,3.5);
\draw (9,2) -- (7,0);
\draw (3,0.) -- (4.5,1.5);
\end{tikzpicture}
\caption{The lines mark the ending frame of the vertex $[i,j]$  in $\AR(\Mod kD_n)$. The area from which there are non-trivial morphisms into $M_{[i,j]}$ is marked in grey}
\label{fig:AR2}
\end{centering}
\end{figure}
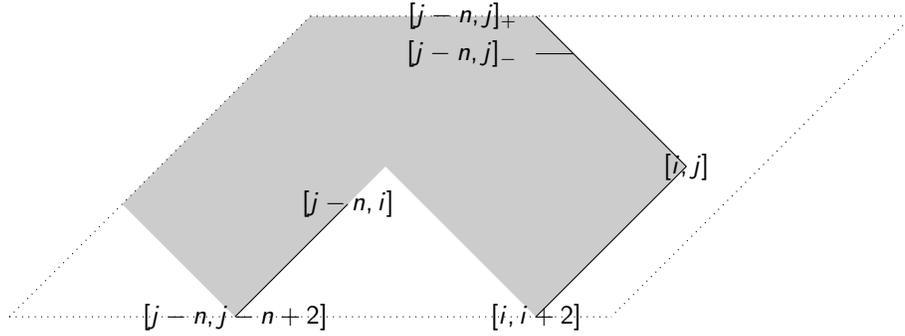

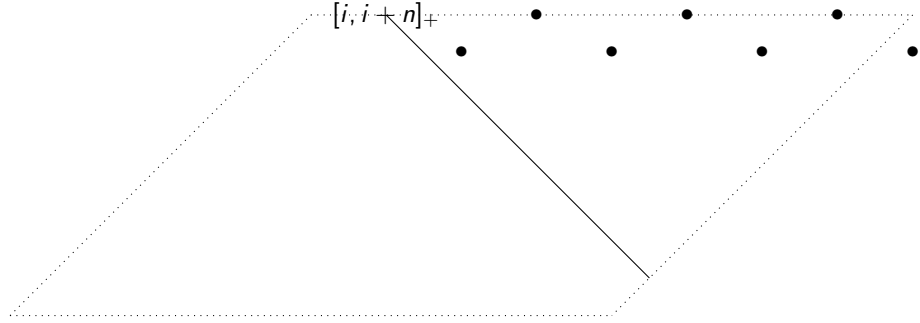
\begin{figure}
\begin{centering}
\begin{tikzpicture} [scale=1, font = \footnotesize, font=\sansmath\sffamily
]

\draw[dotted] (0,0) -- (4,4);
\draw[dotted] (4,4) -- (12,4);
\draw[dotted] (12,4) -- (8,0);
\draw[dotted] (0,0) -- (8,0);

\node (a) at (5,4) {$[i,i+n]_+$};

\node (d) at (6,3.5) {$\bullet$};
\node (e) at (7,4) {$\bullet$};
\node (f) at (8,3.5) {$\bullet$};
\node (g) at (9,4) {$\bullet$};
\node (f) at (10,3.5) {$\bullet$};
\node (f) at (11,4) {$\bullet$};
\node (f) at (12,3.5) {$\bullet$};

\draw (8.5,0.5)-- (5,4);
\end{tikzpicture}
\caption{The line and the bullets mark the starting frame of the signed vertex $[i,i+n]_+$  in $\AR(\Mod kD_n)$}
\label{fig:AR3}
\end{centering}
\end{figure}

\begin{figure}
\begin{centering}
\begin{tikzpicture} [scale=1, font = \footnotesize, font=\sansmath\sffamily
]

\draw[dotted] (0,0) -- (4,4);
\draw[dotted] (4,4) -- (12,4);
\draw[dotted] (12,4) -- (8,0);
\draw[dotted] (0,0) -- (8,0);

\node (a) at (8,4) {$[i,i+n]_+$};
\node (c) at (4,0) {$[i,i+2]$};

\node (e) at (4,4) {$\bullet$};
\node (f) at (5,3.5) {$\bullet$};
\node (g) at (6,4) {$\bullet$};
\node (f) at (7,3.5) {$\bullet$};

\draw (4,0)-- (8,4);

\end{tikzpicture}
\caption{The line and the bullets mark the ending frame of the signed vertex $[i,i+n]_+$  in $\AR(\Mod kD_n)$}
\label{fig:AR4}
\end{centering}
\end{figure}
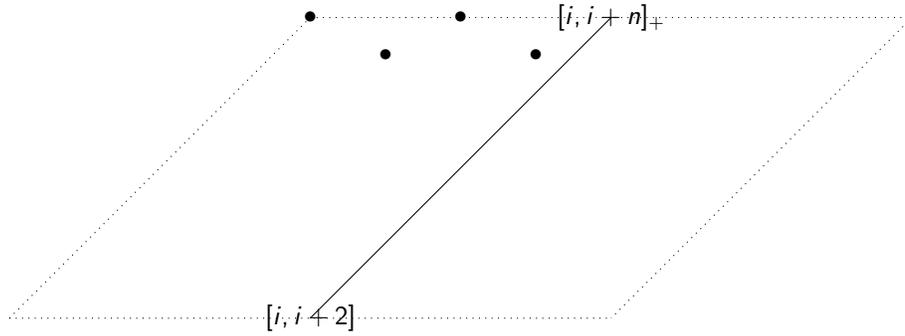

\begin{lemma}\label{L:triangles}
Let $\D$ be a non-crossing diagram of Dynkin type $D_n$ corresponding to the subcategory $D$ of $\CC_{D_n}$. Let $\overline{(x,y)} \in \nc \D \setminus \D$. Then there exists a distinguished triangle
$$
\xymatrix{m_{\overline{(x,y)}} \ar[r]& d \ar[r]& m_{\mu^-_\D(\overline{(x,y)})} \ar[r] & \Sigma m_{\overline{(x,y)}}}
$$
in $\CC_{D_n}$ where $d$ is an object of $D$.
\end{lemma}

Lemma \ref{L:triangles} can be derived from calculations in section 8 of Buan, Marsh, Reineke, Reiten and Todorov's paper \cite{BMRRT}. We provide a proof for the convenience of the reader.

\begin{proof}
First we note how we can use Lemma \ref{intersect frames} to calculate extensions. Assume that the object $m_{b([i,j])}$ in $\CC_{D_n}$ is induced by a projective $kD_n$-module $M_{[i,j]}$ and the object $m_{b([k,l])}$ in $\CC_{D_n}$ is induced by another $kD_n$-module $M_{[k,l]}$, where $b$ is the bijection from Definition \ref{D:bijection}. Then by a result by Buan, Marsh, Reineke, Reiten and Todorov (\cite{BMRRT}, Proposition 1.7(d)) we have
\begin{eqnarray*}
\Ext_{\CC_{D_n}}^1(m_{b([i,j])},m_{b([k,l])}) & \cong &\Ext_{kD_n}^1(M_{[i,j]},M_{[k,l]}) \oplus \Ext_{kD_n}^1(M_{[k,l]},M_{[i,j]})\\
& = & \Ext_{kD_n}^1(M_{[k,l]},M_{[i,j]}).
\end{eqnarray*}
Whenever the elements $b([i,j])$ and $b([k,l])$ in $\A(\p_{2n})$ intersect exactly once, this extension space is one-dimensional and we can apply Lemma \ref{intersect frames} to calculate short exact sequences in $\Mod kD_n$, which induce distinguished triangles in $\CC_{D_n}$. By abuse of notation we introduce additional vertices $[i,i+1] = b^{-1}(\overline{(i,i+1)})$ for pairs of edges $\overline{(i,i+1)} \in \mathcal{E}(\p_{2n})$, to which we associate the zero-module in $\Mod kD_n$, $M_{[i,i+1]} \cong 0$ and the zero-object in $\CC_{D_n}$, $m_{\overline{(i,i+1)}} \cong 0$.

Let now $\overline{(x,y)}$ and $\mu^-_{\D}(\overline{(x,y)})$ be contained in the pair of $\D$-cells $\overline{\langle d_1, \ldots, d_k \rangle}$ with $\overline{(x,y)} = r_{\D}^{-1}(\overline{(d_i,d_j)})$ and $\mu^-_{\D}\overline{(x,y)} = r_{\D}^{-1}(\overline{(d_{i-1},d_{j-1})})$ for some $i,j \in \{1, \ldots, k\}$. The vertices $d_i, d_{j-1},d_j,d_{i-1}$ appear in this order in an anti clockwise direction on the boundary of the $\D$-cell $\langle d_1, \ldots, d_k \rangle$.
Without loss of generality, we may assume that $m_{\overline{(x,y)}}$ is induced by a projective module, i.e.\;$d_i = 0$ and $d_j \in \{2 ,\ldots, n\} \cup \{c\}$, otherwise we obtain the desired distinguished triangle by shifting. 

First consider the case where $m_{\mu^-_{\D}(\overline{(x,y)})}$ is induced by the shift of a projective module, i.e.\;we can assume $d_{i-1} = 2n-1$ and $d_{j-1} \in \{1, \ldots, d_j-1\} \cup \{c\}$.
\begin{itemize}
\item{
Suppose that $\overline{(x,y)} = \overline{(d_i,d_j)}$ is a pair of arcs. We have $0 = d_i < d_{j-1} < d_j < n$. In particular, $\overline{(d_{j-1},d_j)}$ is a pair of arcs.
If $\overline{(d_{j-1},d_j)}$ is a pair of edges, then $m_{\mu_{\D}^-(\overline{(x,y)})} = \Sigma m_{\overline{(x,y)}}$ and we obtain the desired distinguished triangle
$$
m_{\overline{(x,y)}} \to 0 \to m_{\mu^-_{\D}(\overline{(x,y)})} \to \Sigma  m_{\overline{(x,y)}}.
$$
Otherwise, if $\overline{(d_{j-1},d_j)}$ is a pair of internal arcs, then $\overline{(0,d_{j-1}+1)}$ and $\overline{(d_{j-1},d_j)}$ cross precisely once. Intersecting the starting frame of the vertex $[0,d_{j-1}+1]$ with the ending frame of the vertex $[d_{j-1},d_j]$ in $\AR(\Mod kD_n)$ yields the short exact sequence
$$
0 \to M_{[0,d_{j-1}+1]} \to M_{[0,d_j]} \to M_{[d_{j-1},d_j]} \to 0, 
$$
which induces the distinguished triangle
$$
m_{\overline{(0,d_j)}} = m_{\overline{(x,y)}} \to m_{r^{-1}_{\D}\overline{(d_{j-1},d_j)}} \to m_{\overline{(2n-1, d_{j-1})}} = m_{\mu^-_{\D}(\overline{(x,y)})} \to \Sigma  m_{\overline{(x,y)}}
$$
in $\CC_{D_n}$, whose middle term lies in $D$. 
}
\item
{
Now suppose that $\overline{(x,y)} = \overline{(0,n)}$ is a diameter. If $\mu_{\D}^-(\overline{(x,y)})$ is also a diameter, then $m_{\mu_{\D}^-(\overline{(x,y)})} = \Sigma m_{\overline{(x,y)}}$ which yields the distinguished triangle
$$
m_{\overline{(x,y)}} \to 0 \to m_{\mu_{\D}^-(\overline{(x,y)})} \to \Sigma m_{\overline{(x,y)}},
$$
with middle term $0$ in $D$. If on the other hand $\mu_{\D}^-(\overline{(x,y)})$ is a pair of arcs, the pair of arcs $\overline{(0,d_{j-1}+1)}$ crosses the diameter $\overline{(d_{j-1}, d_{j-1}+n)}$ once and  intersecting the starting frame of the vertex $[0,d_{j-1}+1]$ with the ending frame of the vertex $[d_{j-1}, d_{j-1}+n]_\pm$ in $\AR(\Mod kD_n)$ yields the short exact sequence
$$
0 \to M_{[0,d_{j-1}+1]} \to M_{[0,n]_{**}} \to M_{[d_{j-1},d_{j-1}+n]_\pm} \to 0,
$$
where $[0,n]_{**}$ has the same sign as $[d_{j-1},d_{j-1}+n]_{\pm}$ if and only if $d_{j-1} \equiv 0 \mod 2$. Because the diameter $\overline{(x,y)}$ was mutated to a pair of arcs $\overline{(d_{i-1},d_{j-1})}$ we have $d_j = c$ by Remark \ref{diameters to arcs}. Thus the short exact sequence above gives rise to the distinguished triangle
$$
m_{\overline{(0,n)}} = m_{\overline{(x,y)}_{r,g}} \to m_{r_{\D}^{-1}(\overline{(d_{j-1},d_j)}_{r,g})} \to m_{\overline{(2n-1,d_{j-1})}} = m_{\mu^-_{\D}(\overline{(x,y)})} \to \Sigma m_{\overline{(x,y)}_{r,g}},
$$
in $\CC_{D_n}$, where $\overline{(x,y)}_{r,g}$ is of the same colour as $\overline{(d_{j-1},d_j)}_{r,g}$, so the middle term lies in $D$.
}
\end{itemize}
Now consider the case, where $m_{\mu_{\D}^-(\overline{(x,y)})}$ is induced by a $kD_n$-module. By Lemma \ref{dimension of Ext}, the elements $\overline{(x,y)}$ and $\mu_{\D}^-(\overline{(x,y)})$ cross exactly once, so we can apply Lemma \ref{intersect frames} to calculate extensions of $M_{b^{-1}(\mu_{\D}^-(\overline{(x,y)}))}$ by $M_{b^{-1}(\overline{(x,y)})}$ in $\Mod kD_n$.
\begin{itemize}
\item
{Suppose both $\overline{(x,y)}$ and $\mu_{\D}^-(\overline{(x,y)})$ are pairs of arcs. We have $0 = d_i < d_{j-1} < d_j < d_{i-1} < 2n-1$. Because $\overline{(d_i,d_j)}$ and $\overline{(d_{i-1},d_{j-1})}$ intersect exactly once, either $d_j+n \leq d_{i-1} < 2n-1$ or $d_j < d_{i-1} \leq n$. Assume first that $d_j+n \leq d_{i-1} < 2n-1$. We have $b^{-1}(\overline{(d_i,d_j)}) = [0,d_j]$ and $b^{-1}(\overline{(d_{i-1},d_{j-1})}) = [d_{i-1}+n,d_{j-1}+n]$. Intersecting the starting frame of the vertex $[0,d_j]$ with the ending frame of the vertex $[d_{i-1}+n, d_{j-1}+n]$ in $\AR(\Mod kD_n)$ yields the short exact sequence
$$
0 \to M_{[0,d_j]} \to M_{[d_{i-1}+n,n]} \oplus M_{[d_{j-1},d_j]} \to M_{[d_{i-1}+n,d_{j-1}+n]} \to 0,
$$
which induces the distinguished triangle
$$
m_{\overline{(x,y)}} \to m_{r^{-1}_{\D}(\overline{(d_{i-1},d_i)})} \oplus m_{r^{-1}_{\D}(\overline{(d_{j-1},d_j)})} \to \Sigma m_{\overline{(x,y)}}.
$$
The middle term of this distinguished triangle lies in $D$. 

On the other hand, assume  $d_j < d_{i-1} \leq n$. If $d_{i-1} \neq n$, then intersecting the starting frame of the vertex $[0,d_j]$ with the ending frame of the vertex $[d_{j-1},d_{i-1}]$ in $\AR(\Mod kD_n)$ yields the short exact sequence
$$
0 \to M_{[0,d_j]} \to M_{[d_{j-1},d_j]} \oplus M_{[0,d_{i-1}]}\to M_{[d_{j-1},d_{i-1}]} \to 0,
$$
which induces the distinguished triangle
$$
m_{\overline{(x,y)}} \to m_{r_{\D}^{-1}(\overline{(d_{j-1},d_j)})} \oplus m_{r_{\D}^{-1}(\overline{(d_i,d_{i-1})})} \to \Sigma m_{\overline{(x,y)}}
$$
with middle term in $D$. If on the other hand $d_{i-1} = n$, then we obtain the short exact sequence
$$
0 \to M_{[0,d_j]} \to M_{[d_{j-1},d_j]} \oplus M_{[0,n]_+} \oplus M_{[0,n]_-} \to M_{[d_{j-1},n]} \to 0.
$$
This induces the distinguished triangle
$$
m_{\overline{(x,y)}} \to m_{r_{\D}^{-1}(\overline{(d_{j-1},d_j)})} \oplus m_{r_{\D}^{-1}(\overline{(0,n)}_{r})} \oplus m_{r_{\D}^{-1}(\overline{(0,n)}_{g})} \to \Sigma m_{\overline{(x,y)}}.
$$
Because $d_i = 0$ and $d_{i-1} = n$ are neighbouring vertices of the $\D$-cell $\langle d_1, \ldots, d_k \rangle$, the red diameter $\overline{(0,n)}_r$ or the green diameter $\overline{(0,n)}_g$ have to lie in $\D$. Without loss of generality assume that $\overline{(0,n)}_r \in \D$. Then by definition of the replacement map $r_{\D}$, the green diameter $\overline{(0,n)}_g$ is also contained in $\D$. Thus the middle term of the distinguished triangle is an object in $D$.
}
\item
{
Suppose now $\overline{(x,y)}$ is a diameter and $\mu_{\D}^-(\overline{(x,y)})$ is a pair of arcs. By Remark \ref{diameters to arcs} we have $d_j = c$ and further $0 < d_{i-1}+n < d_{j-1} < n$. Let $b^{-1}(\overline{(x,y)}) = [0,n]_{*}$, where the sign depends on the colour of $\overline{(x,y)}$, and $b^{-1}(\overline{(d_{i-1},d_{j-1})}) = [d_{i-1}+n,d_{j-1}+n]$. Intersecting the starting frame of $[0,n]_*$ with the ending frame of $[d_{i-1}+n,d_{j-1}+n]$ yields the short exact sequence
$$
0 \to M_{[0,n]_*} \to M_{[d_{i-1}+n,n]} \oplus M_{[d_{j-1}, d_{j-1}+n]_{**}} \to M_{[d_{i-1}+n,d_{j-1}+n]} \to 0,
$$
where $[d_{j-1},d_{j-1}+n]_{**}$ has the same sign as $[d_i,d_i+n]_*$ if and only if $d_{j-1} \equiv 0 \mod 2$. This induces the distinguished triangle
$$
m_{\overline{(x,y)}_{r,g}} \to m_{r_{\D}^{-1}(\overline{(d_{j-1},d_j)}_{r,g})} \oplus m_{r_{\D}^{-1}(\overline{(d_i,d_{i-1})})} \to m_{\mu_{\D}^-\overline{(x,y)}}\to \Sigma m_{\overline{(x,y)}},
$$
where $\overline{(d_{j-1},d_j)}_{r,g}$ is of the same colour as $\overline{(x,y)}_{r,g}$. The middle term thus lies in $D$. Dual considerations show, that if $\overline{(x,y)}$ is a pair of arcs and $\mu_{\D}^-(\overline{(x,y)})$ is a diameter, we can find the desired distinguished triangle.
}
\item
{
Finally, suppose both $\overline{(x,y)}$ and $\mu_{\D}^-(\overline{(x,y)})$ are diameters. By Lemma \ref{mutation changes colour} they are of different colour. Let $b^{-1}(\overline{(x,y)}) = [0,n]_*$ and $b^{-1}((\mu_{\D}(\overline{(x,y)})))=[d_{i-1},d_{j-1}]_{**}$ be the associated vertices in the Auslander-Reiten quiver $\AR(\Mod kD_n)$. Intersecting the starting frame of the vertex $[0,n]_*$ with the ending frame of the vertex $[d_{j-1},d_{j-1}+n]_{**}$ yields the short exact sequence
$$
0 \to M_{[0,n]_*} \to M_{[d_{j-1},n]} \to M_{[d_{j-1},d_{j-1}+n]_{**}} \to 0,
$$
Because $r^{-1}_{\D}(\overline{(0,d_j)}) = \overline{(x,y)}$ is a diameter we either have $d_j = n$ or $d_j = c$. If $d_j = n$ we obtain $m_{r_{\D}^{-1}(\overline{(d_{j-1},n)})}$ $= m_{\overline{(d_{j-1},d_j)}}$. Otherwise, if $d_j = c$, then in particular $d_{j-1}\neq c$. So  because $\mu_{\D}^-(\overline{(x,y)})$ is a diameter we must have $d_{j-1} = d_{i-1}+n$. Thus we obtain $m_{r_{\D}^{-1}(\overline{(d_{j-1},n)})}= m_{\overline{(0, d_{i-1})}}$. The above short exact sequence induces the distinguished triangle
$$
m_{\overline{(x,y)}} \to m_{r_{\D}^{-1}(\overline{(d_{i-1},d_i)})} \to  m_{\mu_{\D}^-\overline{(x,y)}} \to \Sigma m_{\overline{(x,y)}}
$$
respectively the distinguished triangle
$$
m_{\overline{(x,y)}} \to m_{r_{\D}^{-1}(\overline{(d_{j-1},d_j)})} \to  m_{\mu_{\D}^-\overline{(x,y)}} \to \Sigma m_{\overline{(x,y)}}.
$$
In both triangles the middle term lies in $D$.
}
\end{itemize}
\end{proof}

\begin{theorem}\label{mutations agree general}
Let $\D$ be a non-crossing diagram of Dynkin type $D_n$ corresponding to a rigid subcategory $D$ of $\CC_{D_n}$. Let $\X$ and  $\X'$ be subdiagrams of $\nc \D$ with $\X$ corresponding to the subcategory $X$ and $\X'$ corresponding to the subcategory $X'$ of $\CC_{D_n}$. Then $(\X,\X')$ is a $\D$-mutation pair if and only if $(X,X')$ is a $D$-mutation pair.
\end{theorem}

\begin{proof}

Assume that $(\X,\X')$ is a $\D$-mutation pair. We show that $(X,X')$ is a $D$-mutation pair, i.e.\;that $D \subset X \subset \mu_D(X')$ and $D \subset X' \subset \mu^-_D(X)$, cf.\;Definition \ref{D:mutation}. Because $\D \subset \X,\X'$ we have $D \subset X, X'$. By Lemma \ref{invariant} we have $D = \mu_D(D) \subset \mu_D(X') $ and $D = \mu^-_D(D) \subset \mu^-_D(X)$.

It remains to show that every object $m$ in $X$ which is not an object in $D$ is contained in $\mu^-_D(X)$ and that every object $m'$ in $X'$ which is not an object in $D$ is contained in $\mu_D(X')$. The indecomposable objects in $X$ but not in $D$ are of the form $m_{\overline{(x,y)}}$ with $\overline{(x,y)} \in \X \setminus \D$. The indecomposable objects in $X'$ but not in $D$ correspond to elements in $\X' \setminus \D = \mu^-_{\D}(\X \setminus \D)$ and are of the form $m_{\mu^-_{\D}(\overline{(x,y)})}$ for some $\overline{(x,y)} \in \X \setminus \D$.  

Since both $\overline{(x,y)}$ and $\mu^-_{\D}(\overline{(x,y)})$ lie in $\nc \D$, both  $m_{\overline{(x,y)}}$ and $m_{\mu^-_{\D}(\overline{(x,y)})}$ are objects in $^\perp(\Sigma D) = (\Sigma^{-1} D)^\perp$. Furthermore, by Lemma \ref{L:triangles}, for every element $\overline{(x,y)} \in \nc \D \setminus \D$ there is a distinguished triangle
\begin{eqnarray}\label{tr}
\xymatrix{m_{\overline{(x,y)}} \ar[r] & d \ar[r] & m_{\mu^-_{\D}(\overline{(x,y)})} \ar[r] & \Sigma m_{\overline{(x,y)}} },
\end{eqnarray}
with $d \in D$. This shows that all indecomposable objects in $X'$ lie in $\mu^-_D(X)$ and all indecomposable objects in $X$ lie in $\mu_D(X')$. Since direct sums of distinguished triangles are again distinguished triangles, this proves the first direction of the claim.

On the other hand, suppose $(X,X')$ is a $D$-mutation pair and let $\tilde{X}$ be the subcategory corresponding to the diagram $\mu^-_{\D}(\X)$. Then, because $(\X,\mu^-_{\D}(\X))$ is a $\D$-mutation pair, the pair $(X,\tilde{X})$ is a $D$-mutation pair, therefore $\tilde{X} = \mu^-_D(X) = X'$. So the diagram $\mu^-_{\D}(\X)$ corresponds to $X'$ and we get that $\mu^-_{\D}(\X) = \X'$ and $(\X,\X')$ is a $\D$-mutation pair.
\end{proof}

\begin{corollary} \label{mutations agree}
Let $(X,X^\perp)$ be a torsion pair in the cluster category $\CC_{D_n}$, with torsion part $X$ corresponding to the Ptolemy diagram $\X$ of Dynkin type $D_n$. Let $D \subset X \cap (\Sigma^{-1} X)^\perp$ be a subcategory corrresponding to the subdiagram $\D \subset \X \cap \nc  \X$. Then the $\D$-mutation $\mu_{\D}(\X)$ of $\X$ is a Ptolemy diagram of Dynkin type $D_n$ corresponding to the torsion part $\mu_D(X)$ of the mutated torsion pair $\mu_D(X,X^\perp)$. Similarly, the $\D$-mutation $\mu^-_{\D}(\X)$ of $\X$ is a Ptolemy diagram of Dynkin type $D_n$ corresponding to the torsion part $\mu^-_D(X)$ of the mutated torsion pair $\mu^-_D(X,X^\perp)$.
\end{corollary}

\begin{proof}
Because $\D$ is a subdiagram of $\X \cap \nc \X$, we have $\X \subset \nc \D$. Applying Theorem \ref{mutations agree general} to the $\D$-mutation pairs $(\X, \mu^-_{\D}(\X))$ and $(\mu_{\D}(\X), \X)$ yields the result.
\end{proof}



\end{document}